\documentclass{article}

\usepackage{fullpage}
\newsavebox{\TableThree}
\usepackage{amsmath,amsthm,amssymb,xurl}
\usepackage{mathtools}
\usepackage{xurl}
\usepackage{enumitem} 
\usepackage{xspace} 
\usepackage{xifthen}
\usepackage[hidelinks]{hyperref}
\usepackage{tikz,authblk}
\usepackage{textcomp} 
\usepackage{fancybox}
\usepackage{float}
\usepackage{longtable}
\usepackage{authblk}
\usepackage{xcolor}
\usepackage{soul} 
\usepackage{multicol}
 \usepackage{array}
 \usepackage{ragged2e}
 \usepackage{caption}
\newcolumntype{R}[1]{>{\RaggedRight}p{#1}}
\usepackage[normalem]{ulem}

\usepackage{pseudocode}
\usepackage[mathscr]{euscript}

\usepackage{algorithm}
\usepackage[noend]{algpseudocode}

\usepackage{xparse}

\DeclareDocumentCommand{\NF}{o m m}%
{%
\IfValueTF{#1}
 {\ensuremath{#1\textsc{-nf}(#2,#3)}\xspace}
 {\ensuremath{\textsc{nf}(#2,#3)}\xspace}%
}

\DeclareDocumentCommand{\F}{o}
{%
  \IfValueTF{#1}
    {\ensuremath{\textsc{f}(#1)}\xspace}
    {\ensuremath{\textsc{f}}\xspace}%
}

\newcounter{algcounter}

\algrenewcommand\algorithmicrequire{\textbf{Input:}}
\algrenewcommand\algorithmicensure{\textbf{Output:}}

\usetikzlibrary{shapes.geometric}

\newcommand{\zed}{{\ensuremath{\mathbb{Z}}}}

\newcommand{\lcm}{{\rm lcm}}

\newcommand{\Aut}{{\ensuremath{\mathsf{Aut}}}}

\newcommand{\bb}{\mathbb}
\newcommand{\Z}{\bb{Z}}

\newcommand{\FF}{\bb{F}}

\newcommand{\al}{\alpha}
\newcommand{\de}{\delta}
\newcommand{\eps}{\epsilon}
\newcommand{\es}{\emptyset}
\newcommand{\la}{\lambda}
\newcommand{\ol}{\overline}
\newcommand{\wt}{\widehat}
\newcommand{\wti}{\widetilde}

\newcommand{\bx}{{\bf x}}
\newcommand{\by}{{\bf y}}
\newcommand{\pointy}[1]{{\langle#1\rangle}}

\newcommand{\ord}{\text{ord}}
\newcommand{\sm}{\setminus}
\newcommand{\wh}{\widehat}

\newlength{\boxwidth}
\theoremstyle{plain}
\newtheorem{theorem}{Theorem}[section]
\newtheorem{proposition}[theorem]{Proposition}
\newtheorem{lemma}[theorem]{Lemma}
\newtheorem{corollary}[theorem]{Corollary}
\theoremstyle{definition}
\newtheorem{remark}[theorem]{Remark}
\newtheorem{example}[theorem]{Example}
\newtheorem{definition}[theorem]{Definition}

\floatstyle{plain}
\newfloat{flt}{htbp}{.flt}[section]

\allowdisplaybreaks

\title{The theory of  $(\lambda,\mu)$-fold factorizations of groups}

\author[1]{Donald L.\ Kreher}
\affil[1]{Department of Mathematical Sciences,
Michigan Technological University,
Houghton, MI 49931, U.S.A.}
\author[2]{Shuxing Li}
\affil[2]{Department of Mathematical Sciences, University of Delaware, Newark, DE 19716, U.S.A.}
\author[3]{Douglas R. Stinson}
\affil[3]{David R.\ Cheriton School of Computer Science, University of Waterloo, Waterloo ON, N2L 3G1, Canada}

\date{}

\begin{document}
\maketitle

\begin{abstract}
Motivated by recent work on $\lambda$-fold near-factorizations by Kreher, Li, and Stinson (Designs, Codes and Cryptography, 2026), we introduce and systematically investigate $(\lambda,\mu)$-fold factorizations as combinatorial structures with prescribed representation multiplicities. Such a factorization consists of two subsets of a finite group whose products represent each nonidentity element exactly $\lambda$ times and the identity exactly $\mu$ times. These structures generalize $\lambda$-fold near-factorizations and provide a common framework for several families arising in combinatorial design theory, including strong external difference families with two equal-sized subsets, as well as binary sequence pairs with two-level correlation. Their associated $\{0,1\}$-matrices satisfy the Bridges–Ryser equation $XY=\lambda(J-I)+\mu I$, linking the factorization problem to a classical matrix framework for combinatorial designs. We establish the foundational theory of $(\lambda,\mu)$-fold factorizations, provide two general constructions, and prove several nonexistence results for finite abelian groups.
\end{abstract}

\section{$(\la,\mu)$-fold factorizations of groups}
\label{sec-intro}

Factoring a group into products of subsets provides a powerful way to understand how the group can be decomposed into smaller components, while also revealing relevant structural properties of the subsets involved. Owing to its fundamental nature, group factorization has emerged as a vibrant research area deeply rooted in combinatorial design theory, with broad applications across graph theory, coding theory, functional analysis, and number theory \cite{Szabo04,SS}.

\begin{definition}
\label{Fact.defn}
Let $G$ be a finite multiplicative group with identity $e$.  Suppose that $S, T \subseteq G$. 
The subset pair $(S,T)$ is a \emph{$(\la,\mu)$-fold factorization} of $G$ if the following properties hold:
\begin{enumerate}
\item every non-identity  element $g \in G$ can be expressed as a product $g= xy$ with $x \in S$ and $y \in T$ in exactly $\la$ ways, 
\item the identity element $e$ can be expressed as a product $e= xy$ with $x \in S$ and $y \in T$ in exactly $\mu$ ways. 
\end{enumerate}
$S$ and $T$ are called the \emph{factors} in the factorization.
Moreover, if  $s=|S|$ and $t=|T|$, then the subset pair $(S,T)$ is a $(\la,\mu)$-fold factorization of $G$ of \emph{size $(s,t)$} and it is denoted $\F[G,s,t,\la,\mu]$. 
\end{definition}
Throughout the paper, we  use $e_G$ or  just $e$ to represent the identity element of the group $G$ and $n$ to represent the order of the group $G$. Clearly, an $\F[G,s,t,\la,\mu]$ exists only if the basic counting relation holds:
\begin{equation}
\label{eqn-counting}
st = \la(n-1)+\mu.
\end{equation}
{In Definition \ref{Fact.defn}, we are defining factorizations in multiplicative groups. The definitions of course can be modified  in an obvious way to apply to additive groups, by replacing products by sums in the relevant group. (This comment also applies to various other definitions in this paper involving groups.) We note that many of the constructions and examples in later sections of the paper involve additive abelian groups.}

There has been a substantial body of literature on various structures that are either closely related to, or subfamilies of, $(\la,\mu)$-fold factorizations of finite groups. Specifically, a $(1,1)$-fold factorization is the classical \emph{factorization}, which has been a long-standing research topic \cite{DB53a,DB53b,Dinitz,FG74,H49,H50,Redei,Sands57,Sands79,Szabo17}. A $(1,0)$-fold factorization has been termed a \emph{near-factorization}; these have been recently and historically studied \cite{BHS,DB56,CGHK,KPS2,KMS,Pech03,Pech04}.  A $(\la,0)$-fold factorization is also known as  \emph{$\la$-fold near-factorization}; these were studied in \cite{HMW,KLS,KMS,LM}. Very recently, $\la$-fold near-factorization in association schemes has been investigated in \cite{HLMM}. A special subfamily of $\la$-fold near-factorizations are $(\la,0)$-fold factorizations having  size $(s,s)$. These are equivalent to \emph{strong external difference families} having two subsets of size $s$ \cite{BJWZ,HP,JL19,LM,KPS,PS}, which are motivated by algebraic manipulation detection codes, having applications in cryptography \cite{PS}. Suppose $(S,T)$ is a $(\la,\mu)$-fold factorization of $G$ with $T=\{ g^{-1}: g \in S \}$. Then $S$ is a \emph{difference set} in $G$, which is a classical research object in design theory \cite[Chapter VI]{BJL}.  A $(1,2)$-fold factorization is also known as a \emph{$1$-overlapped factorization}. $1$-overlapped factorizations have been used to construct  thin Lehman matrices \cite{Pace,SS13,SS14,Shi}. In the context of sequence design, $(\la,\mu)$-fold factorizations of the cyclic group $\Z_n$ with $\la \ne \mu$ have been studied under the name of \emph{binary sequence pairs with two-level correlation} \cite{JS08,JS10,LX,PXA}. As a relaxation to the $(1,1)$-fold factorization, the problem of packing a group $G$ with translates of a fixed subset of $G$ has been considered in \cite{B96,B97}.

{In the aforementioned literature and in this paper, the factorization $(S,T)$ requires that the two factors $S$ and $T$ are subsets of the relevant group $G$. We note that \cite{YMSJ} studies  \emph{multifold factorizations} of a group $G$, where two or more factors are allowed in the factorization and each factor can be a multiset of $G$ (i.e., repeated elements are allowed).}

In their classical work \cite{BR}, Bridges and Ryser discussed the following problem:
\begin{quote}
Given positive integers $\lambda$ and $\mu$, 
find nonnegative $n$ by $n$ integral matrices $X$ and $Y$ such that the following matrix equation holds:
\begin{equation}
\label{eqn-matfac}
XY=\la(J_n-I_n)+\mu I_n,
\end{equation}
where $J_n$ denotes an $n \times n$ matrix of ones and $I_n$ denotes an $n \times n$ identity matrix.
\end{quote}
The matrix perspective as phrased in equation \eqref{eqn-matfac} provides a unified framework for many problems in combinatorics. For instance, suppose nonnegative integral matrices $X$ and $Y$ of order $m^2+m+1$ satisfy the matrix equation
$$
XY=(J_{m^2+m+1}-I_{m^2+m+1})+(m+1)I_{m^2+m+1},
$$
such that neither $X$ nor $Y$ is a permutation matrix. Then $X=Y^{T}$ and $X$ is the incidence matrix of a projective plane of order $m$ \cite[Corollary 3.3]{BR}. A related variation of the Bridges--Ryser equation \eqref{eqn-matfac} concerns the equation $XY=Z$, where $X$, $Y$, and $Z$ are adjacency matrices of Cayley graphs on the same finite group $G$ \cite{HM}. Solving this equation matrix equation boils down to finding a factorization of subsets of $G$ \cite[Theorem 2.5]{HM}.

In Section \ref{sec-fundamentals}, we will discuss a matrix perspective of $(\la,\mu)$-fold factorizations of a group $G$, which leads to a factorization of the form \eqref{eqn-matfac} where $X$ and $Y$ are group matrices based on $G$ (group matrices will be defined in Definition \ref{def-groupmatrix}). In fact, an $\F[G,s,t,\la,\mu]$ generates a pair of $\{0,1\}$-matrices that give a representation of a binary \emph{$(n,\mu,\la)$-system on $s$ and $t$}, introduced in \cite[p. 435]{BR}.

Given the elegant connections with existing structures described above, we regard $(\la,\mu)$-fold factorizations of finite groups as natural objects for further study. In this paper, we study $(\la,\mu)$-fold factorizations of finite groups $G$ with $\la \ne \mu$, with a focus on finite abelian groups. Specifically, our contributions include the following:
\begin{enumerate}
\item {For $\la \ne \mu$, we prove that one factor of a  $(\la,\mu)$-fold factorization $(S,T)$ uniquely determines the other (Theorem \ref{thm-unique}). Moreover, knowing one of $S$ and $T$, as well as one of $\la$ and $\mu$, is sufficient to determine the $(\la,\mu)$-fold factorization $(S,T)$ up to complementation (Theorem \ref{thm-unique2}). These results are straightforward generalizations of previous results on $(\la,0)$-fold factorizations \cite[Theorems 1.4, 1.5, 1.6, 1.7]{KLS}.}
\item For every $(\la,\mu)$-fold factorization with $\la \ne \mu$, we establish the commutativity of the two factors (Corollary \ref{cor-commutativity}) by employing a fundamental result of Bridges and Ryser \cite{BR}.
\item After introducing the notion of symmetric factorizations, multipliers of factorizations, and equivalence of
factorizations, we derive a series of necessary and sufficient conditions for a $(\la,\mu)$-fold factorization of a finite abelian group to be translate-equivalent to a symmetric factorization (Theorem \ref{thm-transequivsym}), generalizing a result by de Caen et al.\ concerning $(1,0)$-fold factorizations \cite[Corollary 1]{CGHK}.  
\item We present a mixed-radix construction for $(\la,\la\pm1)$-fold factorizations of cyclic groups (Theorem \ref{thm-mixedradix}), which provides a flexible construction framework for such factorizations. {The mixed-radix construction is a generalization of previously known construction of $(1,2)$-fold and $(1,0)$-fold factorizations of cyclic groups.}
\item We present a product construction that yields a broad class of $(\la,\mu)$-fold factorizations of finite abelian groups (Section \ref{sec-prodcon}).
\item We derive several nonexistence results for $(\la,\mu)$-fold factorizations of finite abelian groups of even order (Section \ref{sec-nonexistence}).
\end{enumerate}

The rest of this paper is organized as follows. In Section~\ref{sec-gp}, we briefly introduce the basic concepts of group rings and character sums that will be used frequently throughout the paper. In Section~\ref{sec-fundamentals}, we establish several fundamental results concerning $(\la,\mu)$-fold factorizations, some of which are derived from a matrix perspective. In Section~\ref{sec-construction}, we present a powerful mixed-radix construction of $(\la,\la\pm1)$-fold factorizations of cyclic groups, as well as a product construction of $(\la,\mu)$-fold factorizations of finite abelian groups. In Section~\ref{sec-nonexistence}, we establish several nonexistence results for $(\la,\mu)$-fold factorizations of finite abelian groups of even order. 
Finally, Section \ref{summary.sec} is a brief summary.

\section{Group rings and character sums}
\label{sec-gp}

In this paper, we adopt  group ring notation, which provides a compact and convenient framework for describing $(\la,\mu)$-fold factorizations of finite groups. Moreover, we employ character sums in conjunction with group rings to extract information about these factorizations. A standard treatment of these topics can be found in \cite[Chapter~1]{Pott95}, while \cite{JL24} offers an introductory account. Below, we provide a brief overview of group rings and character sums.

Let $G$ be a finite multiplicative group with the identity element $e$. The group ring $\Z[G]$ is the $\Z$-module with a basis consisting of all the elements of $G$. Formally,
$$
\Z[G]=\left\{ \sum_{g \in G}a_gg : a_g \in \Z \right\}.
$$
If $X=\sum_{g\in G} a_gg \in \Z[G]$,  
then the $a_g$'s are called the \emph{coefficients} of $X$ and  the coefficient $a_g$ of $g$ in X is denoted by $[X]_g$. (The expression $\sum_{g\in G} a_gg$ is a formal sum of  basis elements of $G$.)

For a subset $S$ of $G$, we denote the group ring element $\sum_{g\in S}g$ by $S$. Therefore, a subset of $G$ is in one-to-one correspondence with a group ring element in $\Z[G]$ with all coefficients equal to either $0$ or $1$. As a ring, the additive and multiplicative operations in $\Z[G]$ are defined as follows. Suppose we have $A,B \in \Z [G]$ with 
$$
A = \sum_{g \in G} a_g g \quad \text{and} \quad B = \sum_{g \in G} b_g g.
$$
The \emph{sum} $A+B$ and the \emph{product} $AB$ are defined to be
\begin{equation}
\label{sum.eq}
A+B=\sum_{g \in G} a_gg + \sum_{g \in G} b_gg =\sum_{g \in G} (a_g+b_g) g
\end{equation}
and
\begin{equation}
\label{product.eq}
AB=\left( \sum_{g \in G} a_gg \right) \left(\sum_{g \in G} b_gg \right)=\sum_{g \in G} \sum_{h \in G}a_{gh^{-1}}b_hg.
\end{equation}
If $A$ and $B$ have non-negative integer coefficients, then they can be viewed as multisets. In this case, $A + B$ denotes the multiset union of $A$ and $B$. 

Moreover, for \[X=\sum_{g\in G} a_gg\in \mathbb{Z}[G],\] we define 
\begin{equation}
\label{inverse.eq}
X^{(-1)}=\sum_{g\in G} a_g g^{-1} \in \Z[G].\end{equation} 
If $X$ is a subset of $G$, then $X^{(-1)}$ denotes the subset of $G$ consisting of all the inverses of elements in $X$. Also, if $G$ is an abelian group, then the group ring $\Z[G]$ is a commutative ring.   

In (\ref{product.eq}) and (\ref{inverse.eq}), we assumed that the group $G$ was written multiplicatively.  However, a group operation might be written additively or multiplicatively.
If $G$ is an additive group, then we would modify the definition of $AB$ and $X^{(-1)}$ in the group ring by replacing products of group elements by sums of group elements, and by replacing inverses of group of elements by additive inverses, i.e., replace $g^{-1}$ by $-g$).

Thus, for an additive abelian group, (\ref{sum.eq}) is unchanged, while (\ref{product.eq}) and (\ref{inverse.eq}) are modified as follows:
$$
AB=\sum_{g \in G} \sum_{h \in G}a_{g-h}b_hg
$$
and 
$$
X^{(-1)}=\sum_{g\in G} a_g (-g).
$$

Suppose $(S,T)$ is a $(\la,\mu)$-fold factorization of $G$. From now on, we identify $S$ and $T$ with the corresponding group ring elements in $\Z[G]$ having $0$ and $1$ coefficients. Therefore $(S,T)$ is a $(\la,\mu)$-fold factorization of $G$ if and only if
\begin{equation}
\label{eqn-def}
ST= \la(G - e)+\mu e \text{\quad in $\zed[G]$}.
\end{equation}

To analyze the group ring equation \eqref{eqn-def}, character theory serves as an indispensable tool. Recall that for a finite multiplicative group $G$, its \emph{exponent}, written as $\exp(G)$, is defined to be the smallest positive integer $b$ fulfilling $g^b = 1_G$ for every $g \in G$. A \emph{character} $\chi$ of {an abelian group} $G$ is defined as a group homomorphism mapping $G$ to the multiplicative group of complex numbers with absolute value $1$. Supposing $G$ has exponent $\exp(G)$, evaluating $\chi$ on any $g \in G$ yields
$$
\chi(g)^{\exp(G)} = \chi(g^{\exp(G)}) = \chi(1_G) = 1.
$$
This demonstrates that every image $\chi(g)$ must be an $\exp(G)$-th root of unity. If $\zeta_{\exp(G)}$ represents a primitive $\exp(G)$-th root of unity, then the image of $\chi$ is contained within the cyclic multiplicative group $\pointy{\zeta_{\exp(G)}}$.

The set of all characters of $G$ forms a group $\wh{G}$ under the operation $(\chi_1 \circ \chi_2)(g) = \chi_1(g) \chi_2(g)$ for all $g \in G$. The identity element in $\wt{G}$ is the \emph{principal character} $\chi_0$, which maps all elements via $\chi_0(g)=1$ for $g\in G$. Any character in $\wt{G}$ distinct from $\chi_0$ is called a \emph{nonprincipal character}. Given $\chi \in \wh{G}$, denote the \emph{order} of $\chi$ by $\ord(\chi)$, which is the smallest positive integer $b$ such that 
$$
\chi^b=\underbrace{\chi \circ \chi \circ \cdots \circ \chi}_\text{$\chi$ occurs $b$ times}=\chi_0.
$$ 
{
If $G$ is a finite abelian group, we have the following explicit description of $\wt{G}$.
\begin{definition}\label{def-charcalc}
Let $G=\Z_{n_1}\times\Z_{n_2}\times\cdots\times\Z_{n_t}$ and let 
$n=\exp(G) = \lcm(n_1, n_2, \dots, n_t)$. 
For each $a=(a_1,a_2,\ldots,a_t) \in G$, 
let $\chi_a : G \rightarrow \pointy{\zeta_n}$ be the function given by 
$$
\chi_a(x) = (\zeta_{n_1}^{a_1})^{x_1} (\zeta_{n_2}^{a_2})^{x_2} \dots (\zeta_{n_t}^{a_t})^{x_t} \quad \text{for each $x=(x_1,x_2,\ldots,x_t) \in G$}. 
$$
\end{definition}
It is straightforward to check that  $\chi_a(x+y) = \chi_a(x) \chi_a(y)$ for all $a,x,y \in G$, and so the function $\chi_a$ is a character of $G$ for each $a \in G$. The character $\chi_a$ is the principal character if and only if $a=0$. Moreover, the set $\wh{G}=\{ \chi_a : a \in G \}$ comprises all $\prod_i n_i$  characters of~$G$, and the \emph{character group} $(\wh{G}, \circ)$ of $G$ is a group whose identity is~$\chi_{0}$, and the groups $G$ and $\wh{G}$ are isomorphic. 
}

For any element $X=\sum_{g\in G}a_gg\in \Z[G]$ and any character $\chi\in\wt{G}$, the \emph{character sum} $\chi(X)$ is given by
$$
\chi(X)=\sum_{g\in G} a_g \chi(g).
$$
Through this evaluation, $\chi$ canonically extends to a ring homomorphism from $\Z[G]$ into the ring of cyclotomic integers $\Z[\zeta_{\exp(G)}]$. Moreover, $$
\chi(X^{(-1)})=\chi \left( \sum_{g\in G}a_gg^{-1} \right)=\sum_{g\in G} a_g \chi(g^{-1})=\sum_{g\in G} a_g\ol{\chi(g)}=\ol{\sum_{g\in G} a_g \chi(g)} =\overline{\chi(X)}.
$$
For a finite abelian group $G$, the following Fourier inversion formula indicates that a group ring element $A \in \Z[G]$ is uniquely determined by its character sums. 

\begin{proposition}[Fourier inversion formula {\cite[Proposition 2.13, Corollary 2.14]{JL24}, \cite[Lemma 1.2.2]{Pott95}}] 
\label{prop-fourier}
Let $G$ be a finite abelian group and let $A= \sum_{g \in G} a_g g\in \Z[G]$. Then
$$
a_g=\frac{1}{|G|}\sum_{\chi \in \wh{G}} \chi(A)\chi(g^{-1}) \quad \text{for each $g \in G$}.
$$
Consequently, for $A, B \in \Z[G]$, $A = B$ if and only if $\chi(A)=\chi(B)$ for every character $\chi \in \wh{G}$.
\end{proposition}

Let $N$ be a subgroup of a finite abelian group $G$. A character $\chi \in \wh{G}$ is \emph{principal} on $N$ if $\chi(g)=1$ for each $g \in N$. Define 
$$
N^\perp=\{ \chi \in \wh{G} : \text{$\chi$ is principal on $N$} \}.
$$
Let $\pi: G \rightarrow G/N$ be the canonical projection. Each character $\chi \in N^\perp$ induces a character $\wti{\chi}$ in $\wh{G/N}$ such that $\chi=\wti{\chi} \circ \pi$. Specifically,
\begin{equation}
\label{eqn-indchar}
\wti{\chi}(gN) = \chi(g) \quad \text{for all $g \in G$}.
\end{equation}
Moreover, the correspondence between $\chi$ and $\wti{\chi}$ induces a bijection between $N^\perp$ and $\wh{G/N}$ \cite[Result 2.10]{JL24}. 

\section{Fundamentals of $(\la,\mu)$-fold factorizations of groups}
\label{sec-fundamentals}

In this section, we describe some fundamental results concerning $(\la,\mu)$-fold factorizations of groups. Some of the results in Section~\ref{sec-fundamentals} are natural generalizations of results in \cite[Sections 1,2]{CGHK}, \cite[Section 2]{KLS}, and \cite[Section 1]{KMS} . 

\subsection{Restricting the range of parameter sets}

In this subsection, we derive some restrictions on the range of parameter sets for $\F[G,s,t,\la,\mu]$. To this end, we identify equivalent parameter sets and exclude certain parameter sets that boil down to existing structures that have already been well studied. 

For an $\F[G,s,t,\la,\mu]$, the following lemma, which is a natural extension of \cite[Lemma 1.1]{KLS}, indicates that the ordering of $s$ and $t$ within the parameter set does not matter. Specifically, $(S,T)$ is a $(\la,\mu)$-fold factorization of $G$  if and only if $(T^{(-1)},S^{(-1)})$  is a $(\la,\mu)$-fold factorization of $G$.

\begin{lemma}[Reflection of $(\la,\mu)$-fold factorizations]
\label{lem-STorder}
$(S,T)$ is an $\F[G,s,t,\la,\mu]$ if and only if $(T^{(-1)},S^{(-1)})$ is an $\F[G,t,s,\la,\mu]$.
\end{lemma}
\begin{proof}
In $\zed[G]$, $ST=\la (G-e)+\mu e$ if and only if  $T^{(-1)}S^{(-1)} = \la(G-e)+\mu e$.
\end{proof}

Let $(S,T)$ be a $(\la,\mu)$-fold factorization of a group $G$. A helpful observation is that, by taking complements of $S$ and/or $T$, {we obtain a collection of four interrelated factorizations of $G$.} The proofs, which are straightforward, are omitted.

\begin{lemma}[Complements of $(\la,\mu)$-fold factorizations]
\label{lem-complement}
{The following statements 
are equivalent:}
\begin{enumerate}
\item $(S,T)$ is an $\F[G,s,t,\la,\mu]$.
\item $(G \sm S,T)$ is an $\F[G,n-s,t,t-\la,t-\mu]$.
\item $(S,G \sm T)$ is an $\F[G,s,n-t,s-\la,s-\mu]$.
\item $(G \sm S,G \sm T)$ is an $\F[G,n-s,n-t,n-s-t+\la,n-s-t+\mu]$.
\end{enumerate}
\end{lemma}

Suppose $(S,T)$ is an  $\F[G,s,t,\la,\mu]$. In view of Lemma \ref{lem-STorder} and Lemma~\ref{lem-complement}, by taking reflections and complements if necessary, we can assume that $1 \le s \le t \le \frac{n}{2}$. 

Note that the parameter sets of each of the four equivalent factorizations in Lemma \ref{lem-complement} must all be nonnegative, so we have $0 \le \mu \le s$. If $\mu=0$, then $(S,T)$ is a $\la$-fold near-factorization of $G$. Similarly, if  $\mu=s$, then $(S,G \sm T)$ is an $(s-\la)$-fold near-factorization of $G$. We refer to \cite{KLS} for an in-depth investigation of near-factorizations and do not consider them in this paper. However, we note that $(\la,\mu)$-fold factorizations of $G$ of size $(s,t)$ with $\mu=s$ are closely related to \emph{difference sets} \cite[Chapter VI]{BJL} and \emph{partial difference sets} \cite{Ma} in $G$. Specifically, combining Lemma \ref{lem-complement} and \cite[Corollaries 2.3, 2.10]{KLS}, we have
\begin{align*}
 \F[G,s,s,\la,s] \text{ exists} & \Leftrightarrow  \F[G,s,n-s,s-\la,0] \text{ exists} \\
                                                  & \Leftrightarrow \text{an $(n,s,\la)$-difference set in $G$ exists, and }                                          \\     \\
 \F[G,s,s+1,\la,s]  \text{ exists} & \Leftrightarrow  \F[G,s,n-s-1,s-\la,0]  \text{ exists} \\ 
                                                    &\Leftrightarrow \text{an $(n,s,\la-1,\la)$-partial difference set in $G$ exists}.
\end{align*}

Moreover, the case $\la=\mu$ is excluded in the current paper, as it corresponds to the well-studied theory of factorizations of finite groups \cite{Szabo04,SS}. Therefore the values of $\mu$ we consider are restricted to $1 \le \mu \le s-1$ and $\mu \ne \la$.

Motivated by the above discussion, for a group $G$ of order $n$, we define a \emph{basic parameter set} to be a tuple $(s,t,\la,\mu)$ that satisfies the following conditions:
\begin{equation}
\label{eqn-basic}
\begin{split}
&st = \la (n-1) + \mu\\
& 1 \le s \le t \le \frac{n}{2}\\
& 1 \le \mu \le s-1 \\
& \la \ne \mu .
\end{split}
\end{equation}
The fundamental problem we study in this paper is the following:  

\begin{quote}
Given a group $G$ of order $n$ and a basic parameter set $(s,t,\la,\mu)$, determine if there is an 
$\F[G,s,t,\la,\mu]$, and if so, classify all $\F[G,s,t,\la,\mu]$ up to equivalence (equivalence of factorizations is defined in Section \ref{equiv.sec}).
\end{quote}

 \bigskip

We have a few additional observations regarding (basic) parameter sets.

\begin{remark} 
\label{rem-basic}
\text{\quad}\vspace{-.2in}\\
\begin{enumerate}
\item It is not difficult to prove that $s \geq 2$ in any basic parameter set. But we do not need to include this restriction as part of the definition.
\item Upper and lower bounds on $\mu$ can be transformed into equivalent upper and lower bounds on $\la$  by using the fundamental equation (\ref{eqn-counting}). We have
$$ 
\la = \frac{st - \mu}{n-1} \quad \text{and} \quad 1 \leq \mu \leq s-1,
$$
so it immediately follows that 
\begin{equation}
\label{lamineq} 
\frac{st-s+1}{n-1} \le \la \le \frac{st-1}{n-1}.
\end{equation}
Of course we do not need to include the inequality \eqref{lamineq} in the equations \eqref{eqn-basic}.
\item Using inequality \eqref{lamineq} and the equations \eqref{eqn-basic}, it is not difficult to show that $\la \leq s-1$. 
\end{enumerate}
\end{remark}

\subsection{A matrix perspective of $(\la,\mu)$-fold factorizations}

In this subsection, we describe a matrix perspective of $(\la,\mu)$-fold factorizations of a (multiplicative) group $G$, establishing the connection of group factorizations to the work of Bridges and Ryser, especially, binary $(n,\mu,\la)$-systems \cite{BR}. For this purpose, we recall the definition of group matrix.

\begin{definition}[group matrix]
\label{def-groupmatrix}
{Let $G$ be a finite multiplicative group of order $n$ with elements $\{g_1, g_2, \ldots,g_n \}$.} An $n \times n$ matrix $M$ with rows and columns indexed by the elements of $G$ is a \emph{group matrix} based on $G$ if for any $g,x,y \in G$, we have $M_{gx,gy}=M_{x,y}$. Equivalently, the $(x,y)$ entry of $M$ is solely determined by the quotient $x^{-1}y$.  
\end{definition}

{It is easy to see that a group matrix corresponding to a cyclic group is a circulant matrix. Moreover, if $g_1 = e$ is the identity element of $G$, then a group matrix $M$ satisfies $M_{g,h}=M_{h^{-1}g,e}$ for all $g,h \in G$. Thus, a group matrix is determined by its first column.} 

For any subset $A \subseteq G$, define an $n \times n$ matrix $M(A)$ as follows:
\begin{equation}
\label{eqn-mat}
\big(M(A)\big)_{g,h} = \begin{cases}
1 & \text{if $g^{-1}h \in A$}\\
0 & \text{otherwise}.
\end{cases}
\end{equation}
Note that, once the labelling of the group elements $g_1, g_2, \ldots , g_n$ of $G$ is fixed, there is a one-to-one correspondence between a subset $A \subseteq G$ and its associated matrix $M(A)$. Below, when we relate the subset $A \subseteq G$ and the matrix $M(A)$, we always assume the underlying labelling of the group elements $g_1, g_2, \ldots , g_n$ of $G$ has been specified, with the convention that $g_1=e$. Throughout this paper, we always use $I$ to denote the $n \times n$ identity matrix and $J$ to denote the $n \times n$ all-ones matrix. 

Clearly $M(A)$ has constant row and column sums equal to $|A|$. Equivalently,
$$
M(A)J=JM(A)=|A|J.
$$
Moreover, {because $g^{-1} h \in A$ if and only if $h^{-1} g \in A^{(-1)}$}, we have
$$
M(A^{(-1)})=M(A)^T.
$$
Under this subset-matrix correspondence,  equation~(\ref{eqn-def}) can be recast into a matrix equation. 
The following theorem, whose proof is omitted, is a natural extension of \cite[p .54]{CGHK}, \cite[Theorem 1.2]{KLS}, and \cite[Theorem 1.3]{KMS}.

\begin{theorem}
\label{thm-matrixrep}
Suppose $S,T \subseteq G$. Then $(S,T)$ is a $(\la,\mu)$-fold factorization of $G$ if and only if 
\begin{equation}
\label{eqn-matrix}
M(S)M(T)=\la (J-I) + \mu I.
\end{equation}
\end{theorem}

Note that $\det(M(S)M(T)) =\det(\la (J-I) + \mu I) = (\mu-\la)^{n-1}\big(\la n+\mu-\la \big)$, as the eigenvalues of $J$ attain the value $0$  $n-1$ times and the value $n$  once. Crucially, the  $\det(M(S)M(T))\neq 0$  if and only if $\la \neq \mu$. This fact distinguishes $(\la,\mu)$-fold factorizations of $G$ into two intrinsically distinct categories, having $\la \ne \mu$ and $\la=\mu$, respectively. Below we will see that these two categories are essentially different. As a preparation, we have the following lemma, which is analogous to \cite[Theorem 1.3]{KLS}; the proof is omitted.

\begin{lemma}
\label{lem-matrixinv}
Let $\la$ and $\mu$ be two integers with $\la \ne \mu$. Let $\la(J - I)+ \mu I = XY$ be a factorization into two $n \times n$ (necessarily invertible) matrices $X$ and $Y$, where $\la(n-1)+\mu \ne 0$. 
\begin{enumerate}
\item Suppose that $XJ = rJ$ for some nonzero $r$. Then
$$
 Y = \frac{\la}{r} J + (\mu-\la) X^{-1}.
$$
\item Suppose that $JY = rJ$ for some nonzero $r$. Then
$$
 X = \frac{\la}{r} J + (\mu-\la) Y^{-1}.
$$
\end{enumerate}
\end{lemma}

Employing Lemma \ref{lem-matrixinv}, it is easy to see that, for each $(\la,\mu)$-fold factorization of the group $G$ with $\la \ne \mu$, one subset uniquely determines the other. The following theorem generalizes \cite[Theorems 1.4 and 1.5]{KLS} and \cite[Theorem 1.5]{KMS}.

\begin{theorem} Let $G$ be a finite group and $\la \ne \mu$.
\label{thm-unique}
\begin{enumerate}
\item Suppose $(S,T)$ and $(S',T)$ are both $(\la,\mu)$-fold factorizations of $G$. Then $S = S'$.
\item Suppose $(S,T)$ and $(S,T')$ are both $(\la,\mu)$-fold factorizations of $G$. Then $T = T'$.
\end{enumerate}
\end{theorem}

\begin{remark}
\label{rem-nonunique}
Suppose $(S,T)$ is a $(\la,\la)$-fold factorization of the group $G$. Then the uniqueness result in Theorem~\ref{thm-unique} does not hold. For instance, suppose that $S$ is a nontrivial subgroup of $G$ having index $r>1$. Let $T$ be any complete set of right coset representatives of $S$ in $G$. Then $ST = G$ and hence $(S,T)$ is a $(1,1)$-fold factorization of the group $G$.
This shows that $T$ is not uniquely determined by $S$, so Theorem \ref{thm-unique} does not hold.
\end{remark}

We record the following fundamental result by Bridges and Ryser, which is a restatement of \cite[Theorem 3.1]{BR} and the last paragraph preceding \cite[Corollary 3.2]{BR}. 

\begin{theorem}
\label{thm-BR}
Let $\la$ and $\mu$ be integers with $\la \notin \{0,\mu\}$. Let $X$ and $Y$ be nonnegative integral matrices of order $n > 1$ such that 
$$
XY=\la (J-I)+\mu I.
$$
Suppose $X$ has constant row and column sum $s$ and $Y$ has constant row and column sum $t$. Then $st=\la(n-1)+\mu$ and 
$$
YX=XY=\la (J-I)+\mu I.
$$
\end{theorem}

Theorem~\ref{thm-unique} indicates that, if we specify $G$, $\la$, and $\mu$, then one factor of a $(\la,\mu)$-fold factorization of $G$ uniquely determines the other. This result can be further strengthened as follows, which can be viewed as an extension of \cite[Theorems 1.6, 1.7]{KLS}, which concerned the case  $\mu=\mu'=0$.

{
\begin{theorem}
\label{thm-unique2}
Suppose $(S,T)$ is a $(\la,\mu)$-fold factorization of $G$ and
$(S',T)$ is a $(\la',\mu')$-fold factorization of $G$, where $S \neq S'$, $\la \neq \mu$ and $\la' \neq \mu'$. Suppose $1 \leq s,s' \leq |G|-1$, where $s = |S|$ and $s' = |S'|$.
Then $S' = G \sm S$, $\la' = t - \la$ and $\mu' = t - \mu$, where $|T| = t$.
\end{theorem}
}
\begin{proof}
We evaluate $M(S)M(T)M(S')$ in two ways.
First, from Theorem \ref{thm-matrixrep}, we have
\[M(S)M(T) = \la (J-I) + \mu I.\]
Hence,
\begin{align}
M(S)M(T)M(S') &= (\la (J-I) + \mu I) M(S')\nonumber\\
&= \la s' J + (\mu - \la) M(S').\label{eq3.9.1}
\end{align}
Next, from Theorem \ref{thm-matrixrep} and Theorem \ref{thm-BR}, we have
\[M(T)M(S') = M(S')M(T) = \la' (J-I) + \mu' I.\]
Hence,
\begin{align}
M(S)M(T)M(S') &= M(S) (\la' (J-I) + \mu' I)\nonumber\\
&= \la' s J + (\mu' - \la') M(S).\label{eq3.9.2}
\end{align}
Equating (\ref{eq3.9.1}) and (\ref{eq3.9.2}), we have
\begin{equation}
\label{eq3.9.3}
 (\mu' - \la') M(S) - (\mu - \la) M(S') = (\la s' - \la' s)J.
\end{equation}
For convenience, denote $N =  (\mu' - \la') M(S) - (\mu - \la) M(S')$.

From (\ref{eq3.9.3}), all the entries of $N$ are identical. 
However, from the definition of $N$, for all $g,h$, we have
\[ N_{g,h} = \begin{cases}
0 & \text{if $M(S)_{g,h} = M(S')_{g,h} = 0$}\\
\la - \mu & \text{if $M(S)_{g,h} = 0$ and $M(S')_{g,h} = 1$}\\
\mu' - \la'  & \text{if $M(S)_{g,h} = 1$ and $M(S')_{g,h} = 0$}\\
\mu'- \mu + \la - \la' & \text{if $M(S)_{g,h} = M(S')_{g,h} = 1$.}
\end{cases}
\]
We refer to the above four possibilities as types, $1,2,3$ and $4$, resp.

We make some observations about the possible types of  the entries in $N$.
\begin{itemize}
\item $N$ cannot contain entries of types 1 and 2 (this is because $\mu \neq \la$).
\item $N$ cannot contain entries of types 1 and 3 (this is because $\mu' \neq \la'$).
\item $N$ cannot contain entries of types 2 and 4 (this is because $\mu' \neq \la'$).
\item $N$ cannot contain entries of types 3 and 4 (this is because $\mu \neq \la$).
\end{itemize}
If all entries of $N$ are of the same type, then $M(S)$ and $M(S')$ are both all-0's or all-1's matrices. We can ignore these cases, due to the assumption that $1 \leq s,s' \leq |G|-1$.
Therefore we can assume that $N$ contains entries of at least two types.  
There are only two cases left to consider.
\begin{description}
\item[Case 1] All entries in $N$ are of type 1 or 4, and entries of both types exist. In this case, we see that $S = S'$, which is ruled out by assumption.
\item[Case 2] All entries in $N$ are of type 2 or 3, and entries of both types exist. Here we have $\la-\mu=\mu'-\la'$ and $S' = G \sm S$, the complement of $S$. Therefore, from Lemma \ref{lem-complement}, $(S',T)$ is an $\F[G,n-s,t,t-\la,t-\mu]$. Thus $s' = n-s$, $\la' = t - \la$ and $\mu' = t - \mu$, where $\la-\mu=\mu'-\la'$ is guaranteed.
\end{description}
Because we have considered all the possible cases, the proof is complete.
\end{proof}

Combining Theorem~\ref{thm-matrixrep} and Theorem~\ref{thm-BR}, we have the following corollary indicating that if $(S,T)$ is a $(\la,\mu)$-fold factorization of $G$ with $\la \notin \{0,\mu\}$, then so is $(T,S)$. Equivalently, the two factors $S$ and $T$ of a $(\la,\mu)$-fold factorization of $G$ with $\la \notin \{0,\mu\}$ are commutative in $\Z[G]$.

\begin{corollary}
\label{cor-commutativity}
Let $G$ be a group (possibly nonabelian), let $\la$ be a positive integer and let $\mu$ be a nonnegative integer with $\la \ne \mu$. If $(S,T)$ is a $(\la,\mu)$-fold factorization of $G$, then $ST=TS=\la(G-e)+\mu e$. Consequently, $(S,T)$ is an $\F[G,s,t,\la,\mu]$ if and only if $(T,S)$ is an $\F[G,t,s,\la,\mu]$.
\end{corollary} 
\begin{proof}
By Theorem~\ref{thm-matrixrep}, if $(S,T)$ is an $\F[G,s,t,\la,\mu]$, then $M(S)M(T)=\la(J-I)+\mu I$, where $M(S)$ has constant row and column sum $s$ and $M(T)$ has constant row and column sum $t$. Because $\la > 0$ and $\la \ne \mu$, by Theorem~\ref{thm-BR}, we have $M(T)M(S)=M(S)M(T)=\la(J-I)+\mu I$, implying that $(T,S)$ is an $\F[G,t,s,\la,\mu]$. 
\end{proof}

Corollary \ref{cor-commutativity} is trivially true for finite abelian groups. The strength of Corollary \ref{cor-commutativity} comes from the fact that it also addresses finite nonabelian groups.

\subsection{Symmetric $(\la,\mu)$-fold factorizations and equivalence of factorizations}
\label{equiv.sec}

For a multiplicative group $G$, a subset $A \subseteq G$ is \emph{symmetric} provided that $g \in A$ if and only if $g^{-1} \in A$. In group ring notation, we have
$A = A^{(-1)}$, indicating that $A$ is invariant under the group inverse operation of $G$. 
Furthermore, if $A$ is symmetric and the identity $e \notin A$, then $M(A)$ is the adjacency matrix of the Cayley graph with connection set $A$.

A subset $A \subseteq G$ has \emph{left multiplier} (resp.\ \emph{right multiplier}) $-1$ if there exists $g \in G$ such that $A^{(-1)}=gA$ (resp.\ $A^{(-1)}=Ag$). A subset $A \subseteq G$ has multiplier $-1$ if $A$ has a left multiplier $-1$ or a right multiplier $-1$. By definition, a symmetric subset of $G$ necessarily has multiplier $-1$ (and $g = e$ in this case).

The following lemma indicates that, for a $(\la,\mu)$-fold factorization, one factor has multiplier $-1$ (resp.\ is symmetric) if and only if the other factor has multiplier $-1$ (resp.\ is symmetric). This result generalizes \cite[Lemma 1.8]{KLS}. 

\begin{lemma}
\label{lem-symmetric}
Let $(S,T)$ be a $(\la,\mu)$-fold factorization of $G$ with $\la \ne \mu$.
\begin{enumerate}
\item If one of the factors $S$ or $T$ has multiplier $-1$, then the other factor also has multiplier $-1$. Specifically, there exists $g \in G$, such that
\begin{enumerate}[label={\textup{(}\alph*\textup{)}}]
\item $S^{(-1)}=Sg$ if and only if $T^{(-1)}=g^{-1}T$
\item $S^{(-1)}=gS$ if and only if $T^{(-1)}=Tg^{-1}$.
\end{enumerate}
\item If one of the factors $S$ or $T$ is symmetric, then the other factor is also symmetric.
\end{enumerate}
\end{lemma}
\begin{proof}
By Lemma \ref{lem-STorder} and Corollary \ref{cor-commutativity},  $(S,T)$ is a $(\la,\mu)$-fold factorization of $G$ with $\la \ne \mu$ if and only if $(S^{(-1)},T^{(-1)})$ is a $(\la,\mu)$-fold factorization of $G$  with $\la \ne \mu$.

For part 1(a), if $S^{(-1)}=Sg$ for some $g \in G$, then $S^{(-1)}g^{-1}T=ST=\la (G - e)+ \mu e=S^{(-1)}T^{(-1)}$. In view of Theorem \ref{thm-unique}, we have $T^{(-1)}=g^{-1}T$. Conversely, if $T^{(-1)}=g^{-1}T$, then $S^{(-1)}g^{-1}T=S^{(-1)}T^{(-1)}=\la (G - e)+ \mu e=ST$, In view of Theorem \ref{thm-unique}, we have $S^{(-1)}=Sg$. Part 1(b) follows from a similar argument. 

Part 2 can be proven using the same argument as in part 1(a) with $g=e$. 
\end{proof}

In view of Lemma \ref{lem-symmetric}, we have the following definition.

\begin{definition}
\label{def-multisym}
\quad 
\begin{enumerate}
\item A $(\la,\mu)$-fold factorization $(S,T)$ of a group $G$ has \emph{multiplier} $-1$ if there exists $g \in G$ such that one of the following holds:
\begin{enumerate}
\item $S^{(-1)}=Sg$ and $T^{(-1)}=g^{-1}T$
\item $S^{(-1)}=gS$ and $T^{(-1)}=Tg^{-1}$.
\end{enumerate}
\item A $(\la,\mu)$-fold factorization $(S,T)$ of a group $G$ is \emph{symmetric} if $S$ and $T$ are both symmetric. 
\end{enumerate}
\end{definition}

\begin{remark}
\label{rem-multi}
When $G$ is a finite abelian group, parts 1(a) and 1(b) in Definition~\ref{def-multisym} are equivalent. Furthermore, in light of Theorem~\ref{thm-unique}, showing that $S^{(-1)} = Sg$ (or $S^{(-1)} = gS$, resp.) suffices to establish part 1(a) (or 1(b), resp.).
\end{remark}

Next, we introduce the notion of equivalence between $(\la,\mu)$-fold factorizations, which naturally generalizes the equivalence notions for near-factorizations and $\la$-fold near factorizations that were studied in \cite[Section 1.3]{KLS}.

\begin{definition}[Equivalence of $(\la,\mu)$-fold factorizations]
\label{def-equivalence}
Suppose that $(S,T)$ and $(S',T')$ are both $\F[G,s,t,\la,\mu]$. {Let $\Aut(G)$ be the automorphism group of $G$.} 
\begin{enumerate}
\item If $s \ne t$, $(S,T)$ and $(S',T')$ are \emph{equivalent} if there exist $f \in \Aut(G)$ and $g \in G$ such that $S' = f(S)g$ and $T' = g^{-1} f(T)$. If $s = t$, the two factorizations are equivalent if $(S,T)$ is equivalent to either $(S',T')$ or $(T',S')$, as defined in the case $s \neq t$.
\item Assume that $G$ is abelian. If $s \ne t$, $(S,T)$ and $(S',T')$ are \emph{translate-equivalent} if there exists $g \in G$ such that $S' = Sg$ and $T' = Tg^{-1}$. If $s = t$, the two factorizations are translate-equivalent if there exists $g\in G$ such that either $S' = Sg$ and $T' = Tg^{-1}$, or $S'=Tg^{-1}$ and $T'=Sg$.
\end{enumerate} 
\end{definition}

Note that part 2 of Definition~\ref{def-equivalence} defines translate equivalence only for factorizations in abelian groups, since the definition relies on commutativity and does not extend to the nonabelian setting. If $\la \ne \mu$, in order to show that $(S,T)$ and $(S',T')$ are equivalent or translate-equivalent $(\la,\mu)$-fold factorizations of $G$, it is sufficient to consider one of the two factors, as the other factor is uniquely determined according to Theorem \ref{thm-unique}.

\begin{lemma}
\label{lem-equivalence}
Suppose $(S,T)$ and $(S',T')$ are both $\F[G,s,t,\la,\mu]$ with $\la \ne \mu$.
\begin{enumerate}
\item If $s \ne t$, then $(S,T)$ and $(S',T')$ are equivalent if there exist $f \in \Aut(G)$ and $g \in G$ such that $S' = f(S)g$ or $T' = g^{-1} f(T)$. If $s = t$, the two factorizations are equivalent if there exist $f \in \Aut(G)$ and $g \in G$ such that $S' = f(S)g$, $S' = f(T)g$, $T' = g^{-1} f(S)$, or $T' = g^{-1} f(T)$.
\item Assume that $G$ is abelian. If $s \ne t$, then $(S,T)$ and $(S',T')$ are translate-equivalent if there exists $g \in G$ such that $S' = Sg$ or $T' = Tg$. If $s = t$, the two factorizations are translate-equivalent if there exists $g \in G$ such that $S' = Sg$, $S' = Tg$, $T' = Sg$, or $T' = Tg$.
\end{enumerate} 
\end{lemma}

In the literature, we have seen close connections between the equivalence notion and the symmetric property of $(\la,0)$-fold factorizations.

\begin{theorem}
\label{thm-equivalence}
\mbox{\quad}
\begin{enumerate}
\item \textup{(de Caen, Gregory, Hughes and Kreher, \cite[Proposition 2]{CGHK})} Let $G$ be a finite abelian group. Let $(S,T)$ be $(1,0)$-fold factorization of $G$. Then there exists $g \in G$, such that $(Sg,Tg^{-1})$ is a symmetric $(1,0)$-fold factorization of $G$.
\item \textup{(Kreher,  Li and Stinson, \cite[Theorem 1.13]{KLS})} Let $G$ be a finite group. Let $(S,T)$ be a $(\la,0)$-fold factorization of $G$. There exist $f \in \Aut(G)$ and $g \in G$ such that $(f(S)g,g^{-1}f(T))$ is a symmetric $(\la,0)$-fold factorization of $G$ if and only if there exists $h \in G$, such that $(Sh,h^{-1}T)$ is a symmetric $(\la,0)$-fold factorization of $G$.
\end{enumerate}
\end{theorem}

We now proceed to generalize Theorem~\ref{thm-equivalence} to $(\la,\mu)$-fold factorizations. The following theorem extends part 2 of Theorem~\ref{thm-equivalence} when $G$ is abelian, where the key insight is that the symmetric property is preserved under group automorphisms. 

\begin{theorem}
\label{thm-transequiv}
Let $(S,T)$ be a $(\la,\mu)$-fold factorization of a finite abelian group $G$. Then $(S,T)$ is equivalent to a symmetric
$(\la,\mu)$-fold factorization of $G$ if and only if $(S,T)$ is translate-equivalent to a symmetric $(\la,\mu)$-fold factorization of $G$.
\end{theorem}
\begin{proof}
The backward direction follows from the definition. For the forward direction, suppose that $(S,T)$ is equivalent to a symmetric factorization. Suppose $s=|S|$ and $t=|T|$. By Definition~\ref{def-equivalence}(1), there exist $f \in\Aut(G)$ and $h\in G$ such that $f(U)h$, $f(V)h^{-1}$ are symmetric, where
$$
(U,V)=\begin{cases}
          (S,T) & \mbox{if $s \ne t$,} \\
          (S,T) \text{ or } (T,S) & \mbox{if $s=t$.}
          \end{cases}  
$$
Applying $f^{-1}$ and setting $g=f^{-1}(h)$ shows that $Ug$ and $Vg^{-1}$ are symmetric. Hence, $(S,T)$ is translate-equivalent to a symmetric $(\la,\mu)$-fold factorization $(Ug,Vg^{-1})$.  
\end{proof}

Theorem~\ref{thm-transequiv} indicates that, if $(S,T)$ is a $(\la,\mu)$-fold factorization of a finite abelian group $G$, then it suffices to inspect all $(\la,\mu)$-fold factorizations of $G$ that are translate-equivalent to $(S,T)$ in order to determine if $(S,T)$ is equivalent to a symmetric $(\la,\mu)$-fold factorization of $G$.

Below, we proceed to generalize part 1 of Theorem~\ref{thm-equivalence}. The following lemma indicates that $(S,T)$ is translate-equivalent to a symmetric $(\la,\mu)$-fold factorization of a finite abelian group $G$ if and only if $(S,T)$ has multiplier $-1$.  

\begin{lemma}
\label{lem-equivmulti}
Suppose $G$ is a finite abelian group and $(S,T)$ is an $\F[G,s,t,\la,\mu]$ with $\la \ne \mu$. 
\begin{enumerate}
\item If $(S,T)$ has multiplier $-1$, then there exists a unique $g \in G$, such that $S^{(-1)}=Sg$ and $T^{(-1)}=Tg^{-1}$. Moreover, there exists $x \in G$ such that $g=x^2$.
\item $(S,T)$ is translate-equivalent to a symmetric $(\la,\mu)$-fold factorization of $G$ if and only if $(S,T)$ has multiplier $-1$.
\end{enumerate}
\end{lemma}
\begin{proof}
For the uniqueness in part 1, it suffices to show that there exists a unique $g \in G$, such that $S^{(-1)}=Sg$.  Assume that there exist $g,h \in G$, such that $S^{(-1)}=Sg=Sh$. {Hence $Sgh^{-1} = S$}. Therefore, for each character $\chi \in \wh{G}$, we have $\chi(S)\chi(gh^{-1})=\chi(S)$. {For each $\chi \in \wh{G}\sm \{ \chi_0 \}$, $\chi(S)\chi(T)=\chi(ST)=\chi(\la G+(\mu-\la)e_G)=\mu-\la \ne 0$. Moreover, $\chi_0(S)=|S| \ne 0$. Thus, for each $\chi \in \wh{G}$, $\chi(S) \ne 0$ and therefore, $\chi(gh^{-1})=1=\chi(e_G)$ for each $\chi \in \wh{G}$. As $\chi$ ranges over $\wh{G}$, the character values $\chi(gh^{-1})$ and $\chi(e_G)$ always agree.} From the Fourier inversion formula, we have $gh^{-1}=e_G$ and thus  $g=h$. 

Define $G^2=\{y^2 : y \in G\}$. Then $G^2$ is a subgroup of $G$ such that $G/G^2$ is isomorphic to an elementary abelian $2$-group. For each $\chi \in (G^2)^\perp$, $\chi$ is of order at most $2$. Thus, for each $g \in G$, $\chi(g) \in \{\pm 1\}$ and $\chi(g^{-1})=\chi(g)$, which implies $\chi(S)=\chi(S^{(-1)})$. Consequently, $\chi(S)=\chi(S^{(-1)})=\chi(S)\chi(g)$ for each $\chi \in (G^2)^\perp$. {Recall that $\chi(S) \ne 0$ for each $\chi \in (G^2)^\perp \subset \wh{G}$, we have $\chi(g)=1$ for each $\chi \in (G^2)^\perp$.} By equation \eqref{eqn-indchar}, each $\wti{\chi} \in \wh{G/G^2}$ satisfies $\wti{\chi}(gG^2)=1$. {Note that $\wti{\chi}(gG^2)=1=\wti{\chi}(G^2)$ for each $\wti{\chi} \in \wh{G/G^2}$. As $\wti{\chi}$ ranges over $\wh{G/G^2}$, the character sums $\wti{\chi}(gG^2)$ and $\wti{\chi}(G^2)$ always agree,} by Fourier inversion formula, we have $gG^2=G^2$ and $g \in G^2$. Hence, there exists $x \in G$ such that $g=x^2$.  

For part 2, if $(S,T)$ has multiplier $-1$, then by part (1), there exists $x \in G$, such that $S^{(-1)}=Sx^2$. Therefore $(Sx)^{(-1)}=Sx$, which implies $Sx$ is symmetric. Thus, in view of Lemma \ref{lem-symmetric}(2), $(S,T)$ is translate-equivalent to a symmetric factorization $(Sx,x^{-1}T)$. Conversely, if $(S,T)$ is translate-equivalent to a symmetric $(\la,\mu)$-fold factorization of $G$, there exists $x \in G$, such that $(Sx)^{(-1)}=Sx$, which implies $S^{(-1)}=Sx^2$. In view of Remark \ref{rem-multi}, $(S,T)$ has multiplier $-1$.
\end{proof}

{
We note that the special case of Lemma \ref{lem-equivmulti} with $(\la,\mu)=(1,0)$, namely,  near-factorizations of a finite abelian group $G$, has been studied in \cite{CGHK}. Specifically, \cite[Corollary 1]{CGHK} proved that every near-factorization of a finite abelian group $G$ necessarily has multiplier $-1$. Furthermore, \cite[Proposition 2]{CGHK} established that every near-factorization of a finite abelian group $G$ is translate-equivalent to a symmetric near-factorization of $G$.}

{
For a $(\la,\mu)$-fold factorization $(S,T)$ of $G$, while the group ring element $ST$ is precisely known, it turns out to be useful and interesting to also examine the structure of the group ring element $S^{(-1)}T$, i.e., in the case of a multiplicative group, we look at quotients instead of products. Note that for two subsets $A$ and $B$ of a group $G$, we have for each $h \in G$,
$$
[AB]_h=|\{ \ell \in G : h\ell^{-1} \in A, \ell \in B \}|=|\{ \ell \in G : \ell \in A^{(-1)}h, \ell \in B \}|=|A^{(-1)}h \cap B|. 
$$
Therefore,
$$
S^{(-1)}T=\sum_{h \in G} [S^{(-1)}T]_h h=\sum_{h \in G} |Sh \cap T| h
$$
encodes the information about the size of the intersection $Sh \cap T$ for each $h \in G$. To illustrate this idea, we have the following example.
}

{
\begin{example}
\label{ex-dev}
Let $G=\Z_7= \{e_G,g,g^2, \dots , g^6\}$ (note that $G$ is written multiplicatively here) . Then $S=\{ e_G, g, g^2 \}$ and $T = \{e_G,g^2,g^3,g^5,g^6\}$ comprise a $(2,3)$-fold factorization of $G$. Given that
$$
S^{(-1)}T=\sum_{h \in G} |Sh \cap T|h=2G+g^5,
$$
we have 
$$
|Sh \cap T|=\begin{cases}
                       3 & \text{if $h=g^5$,} \\
                       2 & \text{otherwise.}
                   \end{cases} 
$$
Consequently, the group ring element $S^{(-1)}T-\la G=S^{(-1)}T-2G$ records the how much the overlaps between $Sh$ and $T$ deviate from the baseline $\la G=2G$.
\end{example}
}

{
Example \ref{ex-dev} motivates the following definition as a core concept in the study of $(\la,\mu)$-fold factorizations. 
}

\begin{definition}[overlap deviation]
\label{def-dev}
Let $(S,T)$ be an $\F[G,s,t,\la,\mu]$ with $\la \ne \mu$.  
Define the \emph{overlap deviation} of the factorization $(S,T)$ of $G$ to be 
\begin{equation}
\label{eqn-dev}
A=S^{(-1)}T-\la G=\sum_{h\in G} a(h) h \in \Z[G].
\end{equation}
\end{definition}

Let $(S,T)$ be an $\F[G,s,t,\la,\mu]$ with $\la \ne \mu$. Set $d=\mu-\la \ne 0$. Then $ST=\la G+de_G$.  The overlap deviation contains useful information concerning a $(\la,\mu)$-fold factorization of a finite abelian group. 

\begin{theorem}
\label{thm-dev}
Let $A=\sum_{h\in G} a(h) h$ be the overlap deviation of a $(\la,\mu)$-fold factorization of a finite abelian group $G$. Then
$$
AA^{(-1)}=d^2 e_G.
$$
Moreover,
\begin{align}
\sum_{h\in G} a(h)    &=d,\label{eqn-suma}\\
\sum_{h\in G} a(h)^2&=d^2, \quad \text{and} \label{eqn-squaresuma}\\
\sum_{h\in G} a(h)a(hx)&=0, \quad \text{for every $x\in G \sm \{ e_G \}$} \notag.
\end{align}
In particular, for every $h \in G$, it holds that $|a(h)| \le |d|$ and $|\{h\in G : a(h)\ne0\}|\le d^2$.
\end{theorem}
\begin{proof}
{
Given that  $G$ is abelian, we have
\begin{align*}
AA^{(-1)}&=(S^{(-1)}T-\la G)(S^{(-1)}T-\la G)^{(-1)} \\
              &=(S^{(-1)}T-\la G)(T^{(-1)}S-\la G) \\
              &=S^{(-1)}TT^{(-1)}S-\la S^{(-1)}TG-\la GT^{(-1)}S+\la^2nG \\
              &=T^{(-1)}S^{(-1)}ST-2\la stG+\la^2nG \\
              &=(\la G+de_G)(\la G+de_G)+(\la^2n-2\la(\la n+d))G \\
              &=\la^2nG+2\la dG+d^2e_G-\la^2nG-2\la dG \\
              &=d^2e_G.
\end{align*}
The remaining equations follow by summing up all the coefficients of $A$ and comparing coefficients in $AA^{(-1)}=d^2 e_G$.}
\end{proof}

\begin{theorem}
\label{thm-transequivsym}
Let $(S,T)$ be a $(\la,\mu)$-factorization of a finite abelian group $G$ where $\la \ne \mu$. The following statements are equivalent.
\begin{enumerate}
\item $(S,T)$ is translate-equivalent to a symmetric factorization.
\item There exists a unique $g \in G$ such that 
$$
S^{(-1)}=Sg, \qquad T^{(-1)}=Tg^{-1}.
$$
\item The overlap deviation of $(S,T)$ is $A=dg$ for the same unique $g \in G$ as in part 2.
\item $d \mid a(h)$ for every $h \in G$. 
\item $a(h)d\ge0$ for every $h \in G$. 
\end{enumerate}
\end{theorem}
\begin{proof}
The equivalence of part 1 and part 2  follows from part 2 of Lemma \ref{lem-equivmulti}. 

Next, from the definition of overlap deviation and Theorem \ref{thm-unique}, we have
$$
A=S^{(-1)}T-\la G=dg \Longleftrightarrow g^{-1}S^{(-1)}T=\la G+de_G \Longleftrightarrow S^{(-1)}=Sg.
$$
Therefore, in view of Remark \ref{rem-multi}, the equivalence of part 2 and part 3 is established.

It is easy to see that part 3 implies part 4. Conversely, assume part 4 holds. Then,  for every $h \in G$, we have
$$
a(h)=dq(h) \quad \text{for some} \quad q(h) \in \Z.
$$
Then equations \eqref{eqn-suma} and \eqref{eqn-squaresuma} give
$$
\sum_{h\in G}q(h)=1 \qquad \sum_{h\in G}q(h)^2=1.
$$
Therefore, there exists a unique $g \in G$, such that
$$
a(h)=\begin{cases}
          d & \text{if $h=g$,} \\
          0 & \text{otherwise.}
        \end{cases}
$$
Hence, $A=dg$, proving part 3. 

It is easy to see that part 3 implies part 5. Conversely, assume that part 5 holds.
For every $h \in G$, define $q(h)=\frac{a(h)}{d}$. Because $a(h)d \ge 0$, it holds that $q(h) \ge 0$. Equations \eqref{eqn-suma} and \eqref{eqn-squaresuma} give
$$
\sum_{h\in G}q(h)=1 \quad\text{and}\quad  \sum_{h\in G} q(h)^2=1.
$$
Therefore, there exists a unique $g \in G$ such that
$$
a(h)=\begin{cases}
          d & \text{if $h=g$,} \\
          0 & \text{otherwise.}
        \end{cases}
$$
Hence, $A=dg$, proving that part 3 holds.
\end{proof}

Corollary \ref{cor-diff1} below is a direct consequence of Theorem \ref{thm-transequivsym} concerning the special case of $d=\mu-\la=\pm 1$. This generalizes de Caen et al.'s result concerning $(1,0)$-fold factorizations \cite[Corollary 1]{CGHK}. We remark that Corollary \ref{cor-diff1} can also be derived following the same approach used in \cite[Theorems 1, 2, Corollary 1]{CGHK}.  

\begin{corollary}
\label{cor-diff1}
Let $G$ be a finite abelian group. Then every $\F[G,s,t,\la,\mu]$ 
with $\mu-\la=\pm 1$ is translate-equivalent to a symmetric 
$\F[G,s,t,\la,\mu]$.
\end{corollary}

\section{Constructions of $(\la,\mu)$-fold factorizations}
\label{sec-construction}

{
In this section, we present some constructions of $(\la,\mu)$-fold factorizations with $\la \ne \mu$. We focus on a general construction in cyclic groups and a product construction in certain finite abelian groups.
We are writing group operations additively in this section.
}

\subsection{Constructions in cyclic groups}
\label{sec-cyclic}

We start with the following folklore construction of $(1,2)$-fold factorizations of $\zed_n$.

\begin{lemma}[Folklore construction of $(1,2)$-fold factorizations]
\label{lem-folklore1}
{Suppose $n+1=st$ with $1 < s,t <n+1$.} Let 
$$
S=\{0,1,\ldots,s-1\} \quad\text{and}\quad T=\{0,s,2s,\ldots,(t-1)s\}.
$$
Then $(S,T)$ is an $\F[\zed_n,s,t,1,2]$.
\end{lemma}

The folklore construction in Lemma \ref{lem-folklore1} can be naturally extended, yielding a $(\la,\la+1)$-fold factorization of $\Z_n$ whenever $\la n+1$ admits an integer factorization $st$ with $1<s,t<n$. 

\begin{theorem}
\label{thm-smallerlambda}
Let $\la \ge 1$. {Suppose $\la n+1=st$ with $1 < s,t <n$}. Let 
$$
S=\{0,1,\ldots,s-1\} \quad\text{and}\quad T=\{0,s,2s,\ldots,(t-1)s\}.
$$
Then $(S,T)$ is an $\F[\zed_{n},s,t,\la,\la+1]$.
\end{theorem}
\begin{proof}
Note that $\gcd(s,n)=1$. From this it follows easily that $T$ is indeed a subset of $\Z_n$. Regarding $S$ and $T$ as subsets of the integers, the sum set $\{ x+y : x \in S, y \in T \}=\{ 0,1, \ldots, \la n \}$. Reducing this sum set modulo $n$, it covers the zero element of $\Z_n$ exactly $\la+1$ times and every nonzero element of $\Z_n$ exactly $\la$ times. Thus,  $(S,T)$ is an $\F[\zed_{n},s,t,\la,\la+1]$.
\end{proof}

Moreover, we have the following folklore construction of $(1,0)$-fold factorizations of $\zed_n$.

\begin{lemma}[Folklore construction of $(1,0)$-fold factorizations]
\label{lem-folklore2}
Suppose $n-1=st$ with $1 < s,t <n-1$. Let 
$$
S=\{0,1,2,\ldots,s-1\} \quad\text{and}\quad T=\{1,s+1,2s+1,\ldots,(t-1)s+1\}.
$$
Then $(S,T)$ is an $\F[\zed_{n},s,t,1,0]$.
\end{lemma}

We can similarly extend the folklore construction in Lemma \ref{lem-folklore2}, yielding a $(\la,\la-1)$-fold factorization of  $\Z_n$ whenever $\la n-1$ admits an integer factorization $st$ with $1<s,t<n$. 

\begin{theorem}
\label{thm-largerlambda}
{Suppose $\la n-1=st$ with $1 < s,t <n$.} Let 
$$
S=\{0,1,2,\ldots,s-1\} \quad\text{and}\quad T=\{1,s+1,2s+1,\ldots,(t-1)s+1\}.
$$
Then $(S,T)$ is an $\F[\zed_{n},s,t,\la,\la-1]$.
\end{theorem}
\begin{proof}
Because $\gcd(s,n)=1$,  $T$ is indeed a subset of $\Z_n$. Regarding $S$ and $T$ as subsets of the integers, the sum set $\{ x+y : x \in S, y \in T \}=\{ 1, 2, \ldots, \la n-1 \}$. Reducing this sum set modulo $n$, it covers the zero element of $\Z_n$ exactly $\la-1$ times and every nonzero element of $\Z_n$ exactly $\la$ times. Thus,  $(S,T)$ is an $\F[\zed_{n},s,t,\la,\la-1]$.
\end{proof}

Very intriguingly, all {known} basic $(\la,\mu)$-fold factorizations of cyclic groups have $\la-\mu=\pm 1$. In \cite{DB56}, De Bruijn proposed an elegant mixed-radix construction of $(1,1)$-fold factorizations of $\Z_n$. This construction has been further extended by Sakuma and Shinohara to construct $(1,0)$- and $(1,2)$-fold factorizations of $\Z_n$ \cite[p. 392]{SS13}. Following this spirit, {we have the following intuitive description of a construction of factorizations based on mixed-radix representations of integers. We note that De Bruijn's construction \cite{DB56} can be viewed as a special case of the following mixed-radix construction.}

Let $r \ge 2$. Set $m_i \geq 2$ for $1 \leq i \leq r$. Suppose that $\prod m_i = \la n + 1$. View an $r$-tuple $(x_1, \dots , x_r)$ as the mixed-radix representation of an integer with radices $m_1, m_2, \dots, m_r$. Thus $(x_1, x_2, \ldots , x_r)$ is the representation of the integer 
$$ 
x_1 + x_2m_1 + x_3m_1m_2 + \: \cdots \: + x_rm_1m_2 \cdots m_{r-1}.
$$
Let 
$$
\mathcal{T} = \{ (x_1, \dots , x_r) : 0 \leq x_i \leq m_i -1\text{ for all } i, 1\leq i \leq r\}.
$$
Clearly, $\mathcal{T} = \{ 0, \dots , \la n \}$. Therefore $\mathcal{T} \bmod n$ consists of $\la$ copies of $\Z_n \sm \{0\}$ and $\la + 1$ copies of $0$.

Let $\es \ne A \subseteq \{1, 2, \dots , r\}$ and let $\es \ne B =  \{1, \dots , r\} \setminus A$. Construct two sets of integers: 
$$
S^* = \{  (x_1, \dots , x_r) : x_i = 0 \text{ for each } i \in A\}
$$
and
$$
T^* = \{  (x_1, \dots , x_r) : x_i = 0 \text{ for each } i \in B\}.
$$
Then the sumset $S^*+T^*$ is equal to $\mathcal{T}$. Define $S = S^* \bmod n$ and $T = T^* \bmod n$. If the modulo $n$ reduction is injective on both $S^*$ and $T^*$, then $(S,T)$ is a $(\la,\la+1)$-fold factorization of $\Z_n$. {With a slight modification, the above construction also applies to the case where $\prod m_i = \la n-1$}.

If either $A$ or $B$ contains two consecutive integers, say $a$ and $a+1$, then the corresponding radices, $m_a$ and $m_{a+1}$, {can be replaced by their product $m_am_{a+1}$ without altering the construction.} Hence, we may assume that both $A$ and $B$ are alternating subsets of $\{1,2,\ldots,r\}$. This observation yields the following mixed-radix construction, {where we handle both the cases $\mu = \la+1$ and $\mu = \la-1$.}

\begin{theorem}[Mixed-radix construction for $(\la,\la\pm1)$-fold factorizations]
\label{thm-mixedradix}
Let $\la$ and $n$ be positive integers, let $\eps\in\{1,-1\}$, and let $r \ge 2$. Suppose that, for $1 \le i \le r$, we have $m_i \ge 2$ and
$$
\prod_{i=1}^{r}m_i=\la n+\eps.
$$
Set $c_{\eps}=\frac{1-\eps}{2}$, $m_0=1$, and for $0 \le j \le r$, let $M_j=\prod_{i=0}^{j}m_i$. Construct two sets of integers: 
\begin{align*}
  S^*&=\Bigg\{ \sum_{\substack{0\le i\le r-1\\ i\ \mathrm{even}}} x_iM_i : 0 \le x_i \le m_{i+1}-1, 0 \le i \le r-1 \Bigg\} \quad \text{and}\\
  T^*&=\Bigg\{ c_{\eps}+\sum_{\substack{0\le i\le r-1\\ i\ \mathrm{odd}}} x_iM_i : 0\le x_i\le m_{i+1}-1, 0 \le i \le r-1 \Bigg\}.
\end{align*}
Let $S=S^* \bmod n$ and $T=T^* \bmod n$. If the modulo $n$ reduction is injective on both $S^*$ and $T^*$, then
$$
ST=\la(\Z_n-e_{\Z_n})+(\la+\eps)e_{\Z_n}.
$$
Hence $(S,T)$ is an $\F[\zed_n,s,t,\la,\la+\eps]$,
where
$$
s=\prod_{\substack{0\le i\le r-1\\ i\ \mathrm{even}}}m_{i+1}
\qquad\text{and}\qquad
t=\prod_{\substack{0\le i\le r-1\\ i\ \mathrm{odd}}}m_{i+1}.
$$
\end{theorem}
\begin{proof}
Note that $M_0=1$ and
$$
M_r-1=\la n+\eps-1.
$$
Every integer $a \in \{0,1,\dots,M_r-1\}$ has a unique mixed-radix representation
$$
a=x_0M_0+x_1M_1+\cdots+x_{r-1}M_{r-1}, \qquad 0 \le x_i \le m_{i+1}-1 \text{ for all }i.
$$
By the definition of $S^*$ and $T^*$, every integer in the set 
$$
\{c_{\eps},c_{\eps}+1,\dots,c_{\eps}+M_r-1\}
$$
has a unique representation as $u+v$ with $u\in S^*$ and $v\in T^*$. Since
$$
\{c_{\eps},c_{\eps}+1,\dots,c_{\eps}+M_r-1\}
=
\begin{cases}
\{0,1,\dots,\la n\},&\text{if } \eps=1,\\
\{1,2,\dots,\la n-1\},&\text{if } \eps=-1,
\end{cases}
$$
every nonzero element of $\Z_n$ occurs exactly $\la$ times after the modulo $n$ reduction, while $e_{\Z_n}$ occurs exactly $\la+\eps$ times. If the modulo $n$ reduction is injective on both $S^*$ and $T^*$, then $S$ and $T$ are two subsets of $\Z_n$ satisfying
$$
ST=\la(\Z_n-e_{\Z_n})+(\la+\eps)e_{\Z_n},
$$
with
$$
|S|=\prod_{\substack{0\le i\le r-1\\ i\ \mathrm{even}}}m_{i+1}
\qquad\text{and}\qquad
|T|=\prod_{\substack{0\le i\le r-1\\ i\ \mathrm{odd}}}m_{i+1}.
$$
\end{proof}

\begin{remark}
\label{rmk-mixedradix}
\mbox{\quad}
\begin{enumerate}
\item {Applying $r=2$ and $\eps \in \{1,-1\}$ in Theorem \ref{thm-mixedradix}, we obtain Theorems \ref{thm-smallerlambda} and \ref{thm-largerlambda}.}
\item The mixed-radix construction encompasses a variety of known constructions for binary sequence pairs with two-level correlation (see \cite[Section 4]{LX}).
\item A slightly more restrictive variant of the mixed-radix construction appears in \cite[Section 1]{SS13} under the name \emph{DBNS factorization}, where it is attributed to earlier work of De~Bruijn \cite{DB56}.
\end{enumerate}
\end{remark}

As we shall see, the mixed-radix construction offers a flexible framework to generate $(\la,\la \pm 1)$-fold factorizations of cyclic groups. As a preparation, we introduce the following invariant to test equivalence of $(\la,\mu)$-fold factorizations of cyclic groups.

\begin{definition}
For a subset $A \subseteq \Z_n$, a divisor $d \mid n$ and $r \in \Z_d$, we define
$$
C_{A,r}(d)=\bigl|\{ a\in A : a\equiv r \pmod d \} \bigr|
$$
to be the number of elements of $A$ in the residue class $r$ modulo $d$. Define the multiset
$$
\mathcal C_A(d)=\{ C_{A,r}(d) :  r\in \Z_d \}.
$$
\end{definition}

{
The following proposition shows that $\mathcal{C}_A(d)$ serves as an invariant for equivalence of $(\lambda,\mu)$-fold factorizations of cyclic groups.
}

\begin{proposition}
\label{prop-C}
Let $(S,T)$ and $(S',T')$ be two $(\la,\mu)$-fold factorizations of $\Z_n$ having size $(s,t)$ and let $d \mid n$.
\begin{enumerate}
\item If $s \ne t$, then $(S,T)$ and $(S',T')$ are inequivalent if $\mathcal C_S(d)\neq\mathcal C_{S'}(d)$ or $\mathcal C_T(d)\neq\mathcal C_{T'}(d)$, for some divisor $d$ of $n$. 
\item If $s=t$, then $(S,T)$ and $(S',T')$ are inequivalent if $\{\mathcal C_S(d),\mathcal C_T(d) \} \ne \{ \mathcal C_{S'}(d),\mathcal C_{T'}(d) \}$, for some divisor $d$ of $n$. 
\end{enumerate}
\end{proposition}
\begin{proof}
{
For a subset $A \subseteq \Z_n$ and $A'=uA+g$ with $u \in\Z_n^*$ and $g \in \Z_n$, it is clear from the definition that $\mathcal C_{A'}(d)=\mathcal C_A(d)$ for every divisor $d$ of $n$. 
}

For part 1, in view of Lemma \ref{lem-equivalence}(1), $(S,T)$ and $(S',T')$ are equivalent if and only if there exist $u \in\Z_n^*$ and $g \in \Z_n$ such that $S' = uS+g$ or $T' = uT-g$. Therefore, $(S,T)$ and $(S',T')$ are equivalent only if $\mathcal C_S(d)=\mathcal C_{S'}(d)$ and $\mathcal C_T(d)=\mathcal C_{T'}(d)$.

Part 2 can be proved similarly.
\end{proof}

Combining Theorem \ref{thm-mixedradix} and Proposition \ref{prop-C}, we have the following example.

\begin{example}\label{ex-Z27}
We apply Theorem \ref{thm-mixedradix} with $n=27$, $\la=3$, $\epsilon=-1$ and the  three sequences  $(m_i)_{i=1}^r$ listed below. We obtain three $\F[\zed_{27},8,10,3,2]$, denoted $(S_i,T_i)$, for $1 \le i \le 3$:
\begin{enumerate}
\item $r=5$, $(m_i)_{i=1}^5=(2,2,2,5,2)$, $S_1=\{0,1,4,5,13,14,17,18\}$,  $T_1=\{0,1,3,6,8,9,11,17,19,25\}$
\item $r=4$, $(m_i)_{i=1}^4=(2,2,4,5)$, $S_2=\{0,1,4,5,8,9,12,13\}$,  $T_2=\{1,3,6,8,11,13,17,19,22,24\}$
\item $r=3$, $(m_i)_{i=1}^3=(2,10,4)$, $S_3=\{0,1,6,7,13,14,20,21\}$, $T_3=\{1,3,5,7,9,11,13,15,17,19\}$.
\end{enumerate}
The following multisets are pairwise distinct:
$$
\mathcal C_{S_1}(9)=\{ 0,0,0,0,1,1,2,2,2 \}, \; \mathcal C_{S_2}(9)=\{ 0,0,0,1,1,1,1,2,2 \}, \; \mathcal C_{S_3}(9)=\{ 0,1,1,1,1,1,1,1,1 \}.
$$
Hence, the factorizations $(S_i,T_i)$, $1 \le i \le 3$, are pairwise inequivalent. 
\end{example}

\subsection{$(\la,\mu)$-fold zero-offset factorizations}
\label{sec-zerooffset}

Inspired by previous research on $(\la,\mu)$-fold factorizations of $\zed_n$ conducted in the context of binary sequence pairs with two-level correlation, we propose the concept of $(\la,\mu)$-fold zero-offset factorizations and present a product construction of such factorizations.

\begin{definition}
Let $A \subseteq \Z_n$. The \emph{binary characteristic vector} of $A$ is $\bx^{A}=(x_i^{A})_{0 \le i \le n-1} \in \{0,1\}^{n}$, where
$$
x_i^{A}=
\begin{cases}
1,& i\in A,\\
0,& i\notin A.
\end{cases}
$$
\end{definition}

Thus every subset of $\Z_n$ determines a unique binary vector of length $n$. Conversely, every binary vector of length $n$ is the characteristic vector of a unique subset of $\Z_n$.

\begin{definition}
Let $\bx=(x_i)_{0 \le i \le n-1}$ and $\by=(y_i)_{0 \le i \le n-1}$ be binary sequences with length $n$. The \emph{correlation} of $\bx$ and $\by$ at shift $0 \le j \le n-1$ is
$$
\theta_{\bx,\by}(j)=\sum_{i=0}^{n-1}(-1)^{x_i+y_{i+j}},
$$
where $i+j$ is regarded as modulo $n$. A binary sequence pair $(\bx,\by)$ has \emph{two-level correlation} if there exist distinct constants $a$ and $b$ such that
$$
\theta_{\bx,\by}(j)=
\begin{cases}
a & \text{if $j=0$,}\\
b & \text{if $1 \le j \le n-1$.}
\end{cases}
$$
\end{definition}

Binary sequence pairs with two-level correlation have been studied in several research papers, including \cite{JS08,JS10,LX,PXA}. They are equivalent to $(\la,\mu)$-fold factorizations of $\zed_n$.

\begin{lemma}
The pair $(S,T)$ is an $\F[\zed_{n},s,t,\lambda,\mu]$ if and only if
\begin{equation}
\label{eqn-translation}
\theta_{{\bx}^{S},\bx^{T^{(-1)}}}(j)=
\begin{cases}
 n-2s-2t+4\mu & \text{if $j=0$,} \\
 n-2s-2t+4\la & \text{if $1 \le j \le n-1$},
\end{cases}
\end{equation}
where $j$ is regarded modulo $n$. Consequently, $(S,T)$ is an 
$\F[\zed_{n},s,t,\lambda,\mu]$
with $\la \ne \mu$ 
if and only if the binary sequence pair $(\bx^{S},\bx^{T^{(-1)}})$ has two-level correlation $n-2s-2t+4\mu$ at shift $0$ and $n-2s-2t+4\la$ elsewhere.
\end{lemma}
\begin{proof}
The entries $x_i^S$ and $x_{i+j}^{T^{(-1)}}$ are both equal to $1$ precisely when $i\in S$ and $i+j\in T^{(-1)}$. Let $c_j$ denote the number of such indices. Hence, 
$$
c_j
=\bigl|\{i\in S:i+j\in T^{(-1)}\}\bigr|
=\bigl|\{i\in S:-i-j\in T\}\bigr|
=\bigl|\{(a,b)\in S\times T:a+b=-j\}\bigr|
=[ST]_{-j}.
$$
To compute the correlation, observe that there are $s$ indices with $x_i^S=1$ and $t$ indices with $x_{i+j}^{T^{(-1)}}=1$,
because translation by $j$ permutes $\mathbb Z_n$. Thus the numbers of indices at which the two entries are $(1,1)$,
$(1,0)$, $(0,1)$, and $(0,0)$ are, respectively, $c_j$, $s-c_j$, $t-c_j$, and $n-s-t+c_j$. The summand $(-1)^{x_i^S+x_{i+j}^{T^{(-1)}}}$ equals $+1$ when the entries agree and $-1$ when they differ. Therefore,
$$
\begin{aligned}
\theta_{\mathbf{x}^S,\mathbf{x}^{T^{(-1)}}}(j)
&=c_j+(n-s-t+c_j)-(s-c_j)-(t-c_j)\\
&=n-2s-2t+4c_j\\
&=n-2s-2t+4[ST]_{-j}.
\end{aligned}
$$
Hence, the equivalence follows from the definitions.
\end{proof}

From the sequence design perspective, binary sequence pairs with two-level correlation and, in particular, sequences having zero correlation at all nonzero shifts, are desirable. In view of equation \eqref{eqn-translation}, such binary sequence pairs satisfy $n-2s-2t+4\la=0$. Below, we define the counterpart of such binary sequence pairs from the viewpoint of $(\la,\mu)$-fold factorizations of finite groups, where the ambient group is not necessarily cyclic.

\begin{definition}
An $\F[G,s,t,\lambda,\mu]$ is called \emph{zero-offset} if
$$
n-2s-2t+4\la=0.
$$
\end{definition}

By definition, only groups of even order can admit zero-offset factorizations. Regarding basic zero-offset factorizations, we have the following tightened restrictions on their parameters. 

\begin{lemma}
\label{lem-zerooffset}
Let $(S,T)$ be a basic zero-offset $\F[G,s,t,\lambda,\mu]$ with $|G| = n$.
Then the following hold.
\begin{enumerate}
\item $(n-2s)(n-2t)=4(\mu-\la)$.
\item $2\la+1 \le s \le t  \le \frac{n}{2}-1$.
\item $\la < \mu$.
\item If $\mu-\la=1$, then $(S,T)$ is necessarily an 
$\F[G,2m-1,2m-1,m-1,m]$, where 
 $|G|=4m$ and $m \ge 2$.
\end{enumerate} 
\end{lemma}
\begin{proof}
For part 1, since $(S,T)$ is zero-offset, in the group ring $\Z[G]$, we have
\begin{align*}
(G-2S)(G-2T)&=G^2-2SG-2GT+4ST\\
                     &=nG-2sG-2tG+4(\la(G-e)+\mu e)\\
                     &=(n-2s-2t+4\la)G+4(\mu-\la)e \\
                     &=4(\mu-\la)e.
\end{align*}
{Applying the principal character on both sides of the above equation} yields $(n-2s)(n-2t)=4(\mu-\la)$. 

For part 2, since $(S,T)$ is basic, we have $1 \le s \le t \le \frac{n}{2}$ and $\la \ne \mu$. If $t=n/2$, then part 1 would force $\la=\mu$, which is impossible. Hence, $t<\frac{n}{2}$. Furthermore, $2s-4\la=n-2t \ge 2$ leads to $s \ge 2\la+1$. Consequently, we have $2\la+1 \le s \le t \le \frac{n}{2}-1$. 

Part 3 follows from parts 1 and 2. 

For part 4, if $\mu-\la=1$, then by part 1, we have $(n-2s)(n-2t)=4$. Since $n-2s$ and $n-2t$ are both even, we have $s=t=\frac{n}{2}-1$. Therefore, $\la=\frac{2s+2t-n}{4}=\frac{n}{4}-1$ and $\mu=\frac{n}{4}$. Since $4 \mid n$, we can write $n=4m$ for $m \ge 1$, which leads to $(s,t,\la,\mu)=(2m-1,2m-1,m-1,m)$. When $m \ge 2$, it is straightforward to verify that the corresponding factorization is basic. When $m=1$, we have $\la=0$, which implies the corresponding factorization is not basic.  
\end{proof}

Below, we describe some examples of zero-offset factorizations.

\begin{example}
\label{ex-seed}
Applying the mixed-radix construction with $r=2$ and $m_1=m_2=2m-1$ generates a basic zero-offset 
$\F[\Z_{4m},2m-1,2m-1,m-1,m]$ for every $m \ge 2$. Note that when $m=1$, the pair $(S,T)$ with $S=T=\{0\} \subset \Z_4$ lead to an $\F[\Z_{4},1,1,0,1]$, which is zero-offset but not basic. Constructions of zero-offset $\F[\Z_{4m},2m-1,2m-1,m-1,m]$ have been presented in \cite[Theorem 6]{JS08},  \cite[Construction A]{JS10}, \cite[Theorem 4]{LX} in the setting of binary sequence pairs with two-level correlation. 
\end{example}

\begin{example}
\label{ex-DS}
Let $D$ be a $k$-element subset of a multiplicative group $G$ of order $n$ with identity $e$. The set $D$ is called a \emph{$(n,k,\la)$-difference set} in $G$ if every nonidentity element $g \in G$ has exactly $\la$ representations of the form
$$
g=d_1d_2^{-1}, \qquad d_1,d_2\in D.
$$
Equivalently,
$$
DD^{(-1)}=\la(G-e)+ke \quad \text{in $\Z[G]$}.
$$
By definition, $(D,D^{(-1)})$ is a 
$\F[G,k,k,\lambda,k]$, which is not basic in view of \eqref{eqn-basic}, because $\mu = s$. Given an $(n,k,\la)$-difference set $D$ in a group $G$, $(D,D^{(-1)})$ is a zero-offset factorization of $G$ if and only if $n=4(k-\la)$. Such difference sets are called \emph{Hadamard difference sets}; they must have parameters $(n,k,\la)=(4N^2,2N^2 \pm N,N^2 \pm N)$ with $N \ge 1$. See \cite[Chapter VI, \S 12]{BJL} for more information on Hadamard difference sets. 

If $D$ is a Hadamard difference set in $G$, then so is the complement $\ol{D}=G \sm D$. Note that the two associated  factorizations $(D,D^{(-1)})$ and $(\ol{D},\ol{D}^{(-1)})$ of $G$ are identical up to complementation. For the purpose of constructing zero-offset factorizations, it suffices to consider Hadamard difference sets with parameters $(4N^2,2N^2-N,N^2-N)$ for some $N \ge 1$. 

A group admitting a Hadamard difference set is called \emph{Hadamard group}. All \emph{known} abelian Hadamard groups are of the following form \cite[Chapter VI, Theorem 12.15]{BJL}:
$$
G=G_2 \times (\Z_{3^{a_1}})^2 \times (\Z_{3^{a_2}})^2 \times \cdots \times (\Z_{3^{a_u}})^2 \times (\Z_{p_1})^{4b_1} \times (\Z_{p_2})^{4b_2} \times \cdots \times (\Z_{p_v})^{4b_v},
$$
where 
\begin{enumerate}
\item $G_2$ is the Sylow $2$-subgroup of nonzero square order satisfying $\exp(G_2) \le 2 \sqrt{|G_2|}$ and $|G_2| \ge 4$
\item $a_i  \ge 0$  for every $1 \le i \le u$, and
\item for each $1 \le i \le v$, $p_i$ is an arbitrary odd prime greater than $3$ and $b_i  \ge 0$.
\end{enumerate}
To sum up, for every  Hadamard group $G$ of order $4N^2$, there exists a zero-offset 
$\F[G,2N^2-N,2N^2-N,N^2-N,2N^2-N]$. 
\end{example}

\subsection{A product construction of $(\la,\mu)$-fold zero-offset factorizations}
\label{sec-prodcon}

In this subsection, we describe a product construction of zero-offset factorizations, which has some similarities to the product construction presented in \cite[Theorem 2.13]{KLS}. 

\medskip

\begin{theorem}[Product construction of zero-offset factorizations]
\label{thm-prod}
For $1 \le i \le 2$, let $G_i$ be a finite group of order $n_i$ with identity $e_i$, and suppose $(S_i,T_i)$ is a zero-offset $\F[G_i,s_i,t_i,\la_i,\mu_i]$.
 Set $\ol{S_i}=G_i\sm S_i$ and $\ol{T_i}=G_i\sm T_i$. In the direct product group $G_1\times G_2$, define
\begin{align*}
 S_3&=(S_1\times\ol{S_2}) \mathbin{\dot\cup} (\ol{S_1}\times S_2),\\
 T_3&=(T_1\times\ol{T_2})  \mathbin{\dot\cup} (\ol{T_1}\times T_2).
\end{align*}
Then $(S_3,T_3)$ is a zero-offset 
$\F[G_1\times G_2,s_3,t_3,\la_3,\mu_3]$, 
 where
\begin{align*}
 s_3&=s_1(n_2-s_2)+s_2(n_1-s_1)\\
 t_3&=t_1(n_2-t_2)+t_2(n_1-t_1)\\
 \la_3&=s_1t_2+s_2t_1-4\la_1\la_2,\\
 \mu_3&=\la_3
 +4(\mu_1-\la_1)(\mu_2-\la_2).
\end{align*}
If both $(S_1,T_1)$ and $(S_2,T_2)$ are basic, then so is $(S_3,T_3)$.
\end{theorem}
\begin{proof}
The fact that $S_3$ and $T_3$ are both subsets of $G_1 \times G_2$ and the expressions of $s_3$, $t_3$ follow easily. For $1 \le i \le 2$, set
$\de_i=\mu_i-\la_i$. Given that
$$
 S_iT_i=\la_iG_i+\de_i e_i.
$$
In view of Lemma \ref{lem-complement} and $(S_i,T_i)$ being zero-offset, we obtain
\begin{align*}
 S_i\ol{T_i}&=(s_i-\la_i)G_i-\de_i e_i,\\
 \ol{S_i}T_i&=(t_i-\la_i)G_i-\de_i e_i,\\
 \ol{S_i}\,\ol{T_i}&=(n_i-s_i-t_i+\la_i)G_i+\de_i e_i=(s_i+t_i-3\la_i)G_i+\de_i e_i.
\end{align*}
Therefore
\begin{align*}
  &S_1T_1\times \ol{S_2}\,\ol{T_2}\\
&=(\la_1G_1+\de_1 e_1) \times ((s_2+t_2-3\la_2)G_2+\de_2 e_2) \\
&=\la_1(s_2+t_2-3\la_2) (G_1 \times G_2)+\de_2\la_1 (G_1 \times \{e_2\})+\de_1(s_2+t_2-3\la_2)(\{e_1\} \times G_2)+\de_1\de_2(e_1,e_2)
\end{align*}
\begin{align*}
  &S_1\ol{T_1}\times \ol{S_2}T_2\\
&=((s_1-\la_1)G_1-\de_1e_1) \times ((t_2-\la_2)G_2-\de_2e_2)\\
&=(s_1-\la_1)(t_2-\la_2)(G_1\times G_2)-\de_2(s_1-\la_1)(G_1\times\{e_2\})-\de_1(t_2-\la_2)(\{e_1\}\times G_2)+\de_1\de_2(e_1,e_2)
\end{align*}
\begin{align*}
  &\ol{S_1}T_1\times S_2\ol{T_2} \\
&=((t_1-\la_1)G_1-\de_1e_1)\times((s_2-\la_2)G_2-\de_2e_2)\\
&=(t_1-\la_1)(s_2-\la_2) (G_1\times G_2)-\de_2(t_1-\la_1)(G_1\times\{e_2\})-\de_1(s_2-\la_2)(\{e_1\}\times G_2)+\de_1\de_2(e_1,e_2)
\end{align*}
\begin{align*}
  &\ol{S_1}\,\ol{T_1}\times S_2T_2 \\
&=((s_1+t_1-3\la_1)G_1+\de_1e_1)\times(\la_2G_2+\de_2e_2)\\
&=\la_2(s_1+t_1-3\la_1) (G_1\times G_2)+\de_2(s_1+t_1-3\la_1)(G_1\times\{e_2\})+\de_1\la_2(\{e_1\}\times G_2)+\de_1\de_2(e_1,e_2).
\end{align*}
Combining the above four terms and simplifying, we have
\begin{align*}
S_3T_3&=S_1T_1\times \ol{S_2}\,\ol{T_2}+S_1\ol{T_1}\times \ol{S_2}T_2+\ol{S_1}T_1\times S_2\ol{T_2}+\ol{S_1}\,\ol{T_1}\times S_2T_2 \\
            &=(s_1t_2+s_2t_1-4\la_1\la_2)(G_1\times G_2)+4\de_1\de_2 (e_1,e_2) \\
            &= \la_3(G_1\times G_2-(e_1,e_2))+\mu_3(e_1,e_2).
\end{align*}
Together with 
\begin{align*}
   &n_1n_2-2s_3-2t_3+4\la_3 \\
 &=n_1n_2-2(s_1(n_2-s_2)+(n_1-s_1)s_2)-2(t_1(n_2-t_2)+(n_1-t_1)t_2)+4(s_1t_2+s_2t_1-4\la_1\la_2) \\
 &=(2s_1+2t_1-4\la_1)(2s_2+2t_2-4\la_2)-2(s_1(s_2+2t_2-4\la_2)+s_2(s_1+2t_1-4\la_1)+t_1(2s_2+t_2-4\la_2)\\
   &+t_2(2s_1+t_1-4\la_1))+4(s_1t_2+s_2t_1-4\la_1\la_2) \\
 &=4(s_1+t_1)(s_2+t_2)-8\la_2(s_1+t_1)-8\la_1(s_2+t_2)-2(2s_1s_2+4s_1t_2+4s_2t_1+2t_1t_2)\\
   &+8\la_2(s_1+t_1)+8\la_1(s_2+t_2)+4s_1t_2+4s_2t_1,  \\
 &=0
\end{align*}
we see that $(S_3,T_3)$ is a zero-offset $\F[G_1\times G_2,s_3,t_3,\la_3,\mu_3]$. It remains to show that if both $(S_1,T_1)$ and $(S_2,T_2)$ are basic, then so is $(S_3,T_3)$. Note that
\begin{align*}
 n_1n_2-2t_3&=(n_1-2t_1)(n_2-2t_2) \ge0 \\
 t_3-s_3&=(t_1-s_1)(n_2-2t_2)+(t_2-s_2)(n_1-2s_1)\ge0
\end{align*}
Therefore, $1 \le s_3 \le t_3 \le \frac{n_1n_2}{2}$. In order to show that $1 \le \mu_3 \le s_3-1$, it suffices to show that $0< \mu_3<s_3$. As $(S_1,T_1)$ and $(S_2,T_2)$ are basic zero-offset factorizations, by Lemma \ref{lem-zerooffset}(3), we have $\la_1 < \mu_1$ and $\la_2 < \mu_2$. Together with $\mu_3-\la_3=4(\mu_1-\la_1)(\mu_2-\la_2)$, we have $\mu_3>\la_3 \ge 0$. Recall that $T_3=(T_1\times\ol{T_2})\mathbin{\dot\cup}(\ol{T_1}\times T_2)$ and $\ol{T_3}=(T_1\times T_2)\mathbin{\dot\cup}(\ol{T_1}\times\ol{T_2})$. Since $\mu_3=[S_3T_3]_{(e_1,e_2)}$, we have $s_3-\mu_3=[S_3\ol{T_3}]_{(e_1,e_2)}$. On the other hand,
\begin{align*}
[S_3\ol{T_3}]_{(e_1,e_2)}&= |(S_1 \times \ol{S_2}) \cap (T_1^{(-1)} \times T_2^{(-1)})|+|(S_1 \times \ol{S_2}) \cap (\ol{T_1^{(-1)}} \times \ol{T_2^{(-1)}})|\\
                                           &+|(\ol{S_1} \times S_2) \cap (T_1^{(-1)} \times T_2^{(-1)})|+|(\ol{S_1} \times S_2) \cap (\ol{T_1^{(-1)}} \times \ol{T_2^{(-1)}})| \\
&=\mu_1(t_2-\mu_2)+(s_1-\mu_1)(n_2-s_2-t_2+\mu_2)+(t_1-\mu_1)\mu_2+(n_1-s_1-t_1+\mu_1)(s_2-\mu_2) \\
&>0.
\end{align*}
Thus, $\mu_3<s_3$. Consequently, $(S_3,T_3)$ is a 
basic factorization. 
\end{proof}

\begin{remark}
A product construction for certain $\la$-fold near-factorizations 
was given in \cite[Section 2.2]{KLS}. Specifically, an 
$\F[G,s,t,\lambda,0]$ where $|G| = n$ and $t=\frac{n-1}{2}$ is called a \emph{half-set near-factorization} of $G$. A half-set near-factorization exists only if $n$ is odd, in which case it must be an
$\F[G,s,\frac{n-1}{2},\lambda,0]$, where $|G| = n$.
 Moreover, an $\F[G,s,t,\lambda,0]$ with $s \le t \le \frac{n-1}{2}$ is a half-set near-factorization if and only if $n-2s-2t+4\la=1$, where $n = |G|$. From this perspective, zero-offset factorizations can be regarded as the analogue of half-set near-factorizations in groups of even order. In \cite[Theorem 2.13]{KLS},  a product construction is presented that combines half-set near-factorizations of two finite groups $G_1$ and $G_2$ into a half-set near-factorization of the direct product group $G_1 \times G_2$. Theorem~\ref{thm-prod} provides a product construction of  zero-offset factorizations following the same spirit of \cite[Theorem 2.13]{KLS}.
\end{remark}

Combining the zero-offset factorizations in Example \ref{ex-seed}, Example \ref{ex-DS}, and the product construction in Theorem \ref{thm-prod}, we have the following construction of zero-offset basic factorizations.

\begin{theorem}
Let $m_i$, $1 \le i \le r$, be positive integers. Let $M_r=\prod_{i=1}^r m_i$ and 
$$
G_r=\Z_{4m_1}\times\Z_{4m_2}\times\cdots\times\Z_{4m_r}.
$$
Let $G$ be a Hadamard group of order $4N^2$ with $N \ge 1$. 
\begin{enumerate}
\item There exists a zero-offset $\F[G_r,s_r,t_r,\lambda_r,\mu_r]$, say $(S_r,T_r)$, where 
\begin{align*}
|G_r| &= n_r= 4^rM_r \\
s_r&= t_r=2^{2r-1}M_r-2^{r-1}  \\
\la_r&= 4^{r-1}M_r-2^{r-1} \\
\mu_r&= 4^{r-1}M_r+4^{r-1}-2^{r-1}.
\end{align*}
Moreover, $(S_r,T_r)$ is basic if and only if $M_r>1$.
\item 
There exists a zero-offset $\F[G_r \times G,s_r',t_r',\lambda_r',\mu_r']$, say $(S_r',T_r')$, where 
\begin{align*}
|G_r \times G| &= n_r'= 4^{r+1}N^2M_r \\
s_r'&= t_r'=2^{2r+1}N^2M_r-2^{r}N  \\
\la_r'&= 4^{r}N^2M_r-2^{r}N \\
\mu_r'&= 4^{r}N^2M_r+4^{r}N^2-2^{r}N.
\end{align*}
Moreover, $(S_r',T_r')$ is basic if and only if $M_r>1$.
\end{enumerate}
\end{theorem}
\begin{proof}
For part 1, we prove the result by induction. The base case of $r=1$ follows from Example \ref{ex-seed}. 
Assume $(S_{r-1},T_{r-1})$  is a zero-offset 
$\F[G_{r-1},s_{r-1},t_{r-1},\lambda_{r-1},\mu_{r-1}]$ where $|G_{r-1}| = n_{r-1}$.
By Example \ref{ex-seed}, there exists a zero-offset 
$\F[\Z_{4m_r},2m_r-1,2m_r-1,m_r-1,m_r]$, say $(S,T)$.
Set 
\begin{align*}
 S_r&=(S_{r-1}\times\ol{S}) \mathbin{\dot\cup} (\ol{S_{r-1}}\times S),\\
 T_r&=(T_{r-1}\times\ol{T})  \mathbin{\dot\cup} (\ol{T_{r-1}}\times T).
\end{align*}
By Theorem \ref{thm-prod}, $(S_r,T_r)$ is a zero-offset 
$\F[G_r,s,t,\la,\mu]$
where $G_r=G_{r-1} \times Z_{4m_r}$, with
\begin{align*}
|G_r|&=4^{r-1}M_{r-1} \cdot 4m_r=4^rM_r=n_r \\
s&=s_{r-1}(4m_r-(2m_r-1))+(2m_r-1)(4^{r-1}M_{r-1}-s_{r-1})=2^{2r-1}M_r-2^{r-1}=s_r \\
t&=t_{r-1}(4m_r-(2m_r-1))+(2m_r-1)(4^{r-1}M_{r-1}-t_{r-1})=2^{2r-1}M_r-2^{r-1}=t_r \\
\la&=s_{r-1}(2m_r-1)+(2m_r-1)t_{r-1}-4 \la_{r-1}(m_r-1)=4^{r-1}M_r-2^{r-1}=\la_r \\
\mu&=\la+4(\mu_{r-1}-\la_{r-1})=4^{r-1}M_r+4^{r-1}-2^{r-1}=\mu_r .
\end{align*}
Finally, note that $1 \le s_r=t_r < \frac{n_r}{2}$ and $\la_r \ne \mu_r$. Then $(S_r,T_r)$ is basic if and only if $\mu_r<s_r$, which is equivalent to $M_r>1$. 

By applying the product construction to the zero-offset factorizations from part 1 and Example \ref{ex-DS}, we obtain part 2.  
\end{proof}


Finally, we describe an example of $(\la,\mu)$-factorization in $\Z_5^2$. This example is noteworthy as it is so far the only {known} basic $(\la,\mu)$-factorization of a noncyclic abelian group of odd order. 

\begin{example}
\label{ex-Z_5^2}
Let $\al$ be a primitive element of $\FF_{25}$ satisfying $\al^2+4\al+2=0$. Let $C_0=\langle \al^4 \rangle$ and for $0 \le i \le 3$, $C_i=\al^iC_0$. Set 
\begin{align*}
S_1=\{0\} \cup C_0, &\quad  T_1=\{0\} \cup C_2  \\
S_2=\{0\} \cup C_1, &\quad  T_2=\{0\} \cup C_3 
\end{align*}
Then $(S_1,T_1)$ and $(S_2,T_2)$ are both $\F[\Z_5^2,7,7,2,1]$. Moreover, a computation indicates that, up to equivalence, there is a unique $\F[\Z_5^2,7,7,2,1]$. 
\end{example}

\section{Some nonexistence results for  $(\la,\mu)$-fold factorizations} 
\label{sec-nonexistence}

In this section, $G$ is a finite abelian group of order $n$ and $(S,T)$ is an $\F[G,s,t,\la,\mu]$ with $\la \ne \mu$. Hence, $ST=\la G+de_G$ with $d=\mu-\la \ne 0$. Here, we use multiplicative notation for the group operation.

Recall that $G^2=\{y^2 : y \in G\}$ is a subgroup of $G$ and $G/G^2$ is isomorphic to an elementary abelian $2$-group. Let $G_2$ be the Sylow $2$-subgroup of $G$ with rank $r$. Then, $|G/G^2|=2^{r}$. Specifically, 
$$
\begin{cases}
 \text{$G^2 \lneq G$, $G/G^2$ is isomorphic to the elementary abelian $2$-group of order $2^r$} & \text{if $|G|$ is even}    \\
 \text{$G^2=G$, $G/G^2$ is a trivial group} & \text{if $|G|$ is odd.}              
\end{cases}
$$
Note that the subgroup $(G^2)^\perp=\{\chi \in \wh{G} : \text{$\chi$ is principal on $G^2$} \}$ consists precisely of characters of $G$ of order at most $2$. Therefore, for each $\chi \in (G^2)^\perp$, $\chi(g) \in \{ \pm 1\}$ for each $g \in G$. Set
$$
|(G^2)^\perp|=|G/G^2|=2^r \qquad \text{and} \qquad u=|G^2|=\frac{n}{2^r}.
$$
Let $\{ C_0=G^2, C_1, \dots, C_{2^r-1} \}$ denote the complete set of cosets of $G^2$ in $G$. For each $0 \le j \le 2^r-1$, define
$$
s_j=|S \cap C_j| \qquad \text{and} \qquad t_j=|T \cap C_j|,
$$
where $0 \le s_j,t_j \le u$, $\sum_{j=0}^{2^r-1} s_j=s$ and $\sum_{j=0}^{2^r-1} t_j=t$. Note that each character $\chi \in (G^2)^\perp \sm \{\chi_0\}$ has kernel $H=\ker \chi$ and
\begin{equation}
\label{eqn-realchival}
\chi(S)=|S \cap H|-|S \sm H|.
\end{equation}

Below, we derive some restrictions on the $(\la,\mu)$-fold factorization of finite abelian groups of even order. 

\begin{theorem}
\label{thm-divisibility}
Let $G$ be a finite abelian group of even order. Let $(S,T)$ be an $\F[G,s,t,\la,\mu]$ with $\la \ne \mu$. Then, for each character $\chi \in (G^2)^\perp \sm \{\chi_0\}$, the following hold.
\begin{enumerate}
\item $\chi(S)$ and $\chi(T)$ are nonzero integers satisfying $\chi(S)\chi(T)=d$. In particular, $\chi(S)$ and $\chi(T)$ are divisors of $d$.
\item $|\chi(S)| \le s$ and $\chi(S) \equiv s \pmod 2$, as well as $|\chi(T)| \le t$ and $\chi(T) \equiv t \pmod 2$.
\item If $s$ and $t$ are both even, then $4 \mid d$.
\end{enumerate}
\end{theorem}
\begin{proof}
Let $\chi \in (G^2)^\perp \sm \{\chi_0\}$. As $\chi$ is of order at most $2$, it takes values in $\{\pm 1\}$. Therefore, both $\chi(S)$ and $\chi(T)$ are integers. 

For part 1, we have $\chi(S)\chi(T)=\chi(ST)=\chi(\la G+de_G)=d \ne 0$, and the rest follows.

For part 2, let $H=\ker \chi$. By equation \eqref{eqn-realchival} together with $|S \cap H|+|S \sm H|=s$, we have
$$
\chi(S)=s-2|S \sm H|,
$$
whence $|\chi(S)| \le s$ and $\chi(S) \equiv s \pmod 2$. Similarly, $|\chi(T)| \le t$ and $\chi(T) \equiv t \pmod 2$.

Part 3 follows from parts 1 and 2. If $s$ and $t$ are both even, then $\chi(S)$ and $\chi(T)$ are both even, and therefore $4 \mid \chi(S)\chi(T)=d$.
\end{proof}

In part 3 of Theorem \ref{thm-divisibility}, if $4 \mid n$ and $s,t$ are both even, then $4 \mid st=\la n+d$ already implies $4 \mid d$. Hence, part 3 of Theorem \ref{thm-divisibility} provides new information exactly when $n \equiv 2 \pmod 4$.

By applying the Fourier inversion formula to the group of characters $(G^2)^\perp$, we have the following theorem.

\begin{theorem}
\label{thm-moment}
Let $G$ be a finite abelian group of even order. Let $|G/G^2|=2^{r}$. Let $(S,T)$ be an $\F[G,s,t,\la,\mu]$ with $\la \ne \mu$. Then the following hold.
\begin{enumerate}
\item
$$
\sum_{\chi \in (G^2)^\perp \sm \{\chi_0\}} \chi(S)^2=2^r\sum_{j=0}^{2^r-1}s_j^2-s^2 \qquad \text{and} \qquad \sum_{\chi \in (G^2)^\perp \sm \{\chi_0\}} \chi(T)^2=2^r\sum_{j=0}^{2^r-1}t_j^2-t^2.
$$
\item
$$
2^r\sum_{j=0}^{2^r-1}s_j^2-s^2 \le (2^r-1)d^2 \qquad \text{and} \qquad 2^r\sum_{j=0}^{2^r-1}t_j^2-t^2 \le (2^r-1)d^2.
$$
\item Let $\ol{s}$ and $\ol{t}$ denote the least nonnegative residues of $s$ and $t$ modulo $2^r$. Then
$$
\ol{s}\,(2^r-\ol{s}) \le (2^r-1)d^2 \qquad \text{and} \qquad \ol{t}\,(2^r-\ol{t}) \le (2^r-1)d^2.
$$
\end{enumerate}
\end{theorem}
\begin{proof}
It suffices to prove the corresponding results for $S$, as the analogous results for $T$ follow similarly. Set $K=G/G^2$. Let $\pi \colon G \to K$ be the canonical projection. Each $\chi \in (G^2)^\perp$ induces a unique $\wti{\chi} \in \wh{K}$ via $\chi=\wti{\chi} \circ \pi$, and this correspondence is a bijection between $(G^2)^\perp$ and $\wh{K}$. Recall that $\{ C_0=G^2, C_1, \dots, C_{2^r-1} \}$ denotes the complete set of cosets of $G^2$ in $G$. For each $0 \le j \le 2^r-1$, let
$$
s_j=|S \cap C_j|.
$$
Setting $\ol{S}=\pi(S)=\sum_{j=0}^{2^r-1}s_j \pi(C_j) \in \Z[K]$, we have $\chi(S)=\wti{\chi}(\ol{S})$ for each $\chi \in (G^2)^\perp$. 
Noting that $K$ is an elementary abelian $2$-group and applying the Fourier inversion formula, we have
$$
\sum_{\chi \in (G^2)^\perp} \chi(S)^2=\sum_{\wti{\chi} \in \wh{K}} \wti{\chi}\bigl(\ol{S}\,\ol{S}\bigr)=|K| \cdot \bigl[\ol{S}\,\ol{S}\bigr]_{C_0}=2^r\sum_{j=0}^{2^r-1}s_j^2.
$$
Together with $\chi_0(S)^2=s^2$, where $\chi_0$ is the principal character of $G$, we have 

\begin{equation}
\label{eqn-Smoment}
\sum_{\chi \in (G^2)^\perp \sm \{\chi_0\}} \chi(S)^2=2^r\sum_{j=0}^{2^r-1}s_j^2-s^2,
\end{equation}
establishing part 1.

{By Theorem \ref{thm-divisibility}(1), for each $\chi \in (G^2)^\perp \sm \{\chi_0\}$, $\chi(S)$ is an integer dividing $d$. Therefore, $|\chi(S)| \le |d|$ for each $\chi \in (G^2)^\perp \sm \{\chi_0\}$.} Together with equation \eqref{eqn-Smoment}, we derive part 2. 

Write $s=q2^r+\ol{s}$ with $q \ge 0$ and $0 \le \ol{s} \le 2^r-1$. Given that $\sum_{j=0}^{2^r-1} s_j=s$, the sum $\sum_{j=0}^{2^r-1} s_j^2$ is minimized when the $s_j$ are as equal as possible. Therefore,
$$
2^r\sum_{j=0}^{2^r-1} s_j^2-s^2 \ge 2^r\bigl[\ol{s}\,(q+1)^2+(2^r-\ol{s})q^2\bigr]-(2^rq+\ol{s})^2=\ol{s}\,(2^r-\ol{s}).
$$
Together with part 2, we derive part 3.
\end{proof}

As a corollary, we have the following result concerning elementary abelian $2$-groups. 

\begin{corollary}
\label{cor-elemabnonexist}
Let $m \ge 1$ and let $G=\Z_2^m$. If $d=\mu-\la \in \{\pm 1,\pm 2\}$, then there does not exist a basic $(\la,\mu)$-fold factorization of $G$.
\end{corollary}
\begin{proof}
Assume for contradiction that $(S,T)$ is a basic $\F[G,s,t,\la,\mu]$. Since $G=\Z_2^m$, we have $G^2=\{e_G\}$ and every $\chi \in \wh{G}$ is of order at most $2$. In this case, $s_j \in \{0,1\}$ and $\sum_{j=0}^{2^r-1} s_j^2=\sum_{j=0}^{2^r-1} s_j=s$. Then equation \eqref{eqn-Smoment} leads to 
\begin{equation}
\label{eqn-elemab}
\sum_{\chi \in \wh{G} \sm \{\chi_0\}} \chi(S)^2=s(2^m-s).
\end{equation}

We consider the two cases of $|d|=1$ and $|d|=2$ separately. 
If $|d|=1$, then by part 1 of Theorem \ref{thm-divisibility}, $\chi(S)^2=1$. Equation \eqref{eqn-elemab} now yields
$$
s(n-s)=\sum_{\chi \in \wh{G} \sm \{\chi_0\}} \chi(S)^2=n-1,
$$
which implies $s \in \{1,2^m-1\}$. In view of Remark \ref{rem-basic}, this contradicts $(S,T)$ being basic.

If $|d|=2$, then by part 1 of Theorem \ref{thm-divisibility}, $|\chi(S)| \cdot |\chi(T)|=2$ for each $\chi \in \wh{G} \sm \{\chi_0\}$, and consequently $\{|\chi(S)|, |\chi(T)|\}=\{1,2\}$. By part 2 of Theorem \ref{thm-divisibility}, $\chi(S) \equiv s \pmod 2$ and $\chi(T) \equiv t \pmod 2$ for each $\chi \in \wh{G} \sm \{\chi_0\}$.  

If $s$ is odd, then $\chi(S)$ is odd and $\chi(S)^2=1$ for each $\chi \in \wh{G} \sm \{\chi_0\}$. As above, equation \eqref{eqn-elemab} gives $s(n-s)=n-1$ and $s \in \{1,2^m-1\}$, contradicting $(S,T)$ being basic.

If $s$ is even, then $t$ is odd. A similar argument shows $t \in \{1,2^m-1\}$, contradicting $(S,T)$ being basic.
\end{proof}



There are  $26$ basic parameter sets whose nonexistence follows from the results of this section. They are presented in Table \ref{tab-excluded}.

\begin{table}[tb]
\caption{Nonexistence of factorizations in abelian groups of order at most $32$.}
\label{tab-excluded}
\begin{center}
\begin{tabular}{|c|c|c|c|c|}
\hline
$G$ & $2^r$ & $(s,t,\la,\mu)$ & $d$ & authority \\
\hline
$\zed_{2}^{3}$ & $8$ & $(3,3,1,2)$ & $+1$ & Corollary \ref{cor-elemabnonexist} \\
$\zed_{14}$ & $2$ & $(4,4,1,3)$ & $+2$ & Theorem \ref{thm-divisibility}(3) \\
$\zed_{2}^{4}$ & $16$ & $(7,7,3,4)$ & $+1$ & Corollary \ref{cor-elemabnonexist} \\
$\zed_{22}$ & $2$ & $(4,6,1,3)$ & $+2$ & Theorem \ref{thm-divisibility}(3) \\
$\zed_{22}$ & $2$ & $(8,8,3,1)$ & $-2$ & Theorem \ref{thm-divisibility}(3) \\
$\zed_{6} \times \zed_{2}^{2}$ & $8$ & $(5,5,1,2)$ & $+1$ & Theorem \ref{thm-moment}(3) \\
$\zed_{6} \times \zed_{2}^{2}$ & $8$ & $(11,11,5,6)$ & $+1$ & Theorem \ref{thm-moment}(3) \\
$\zed_{26}$ & $2$ & $(8,10,3,5)$ & $+2$ & Theorem \ref{thm-divisibility}(3) \\
$\zed_{30}$ & $2$ & $(4,8,1,3)$ & $+2$ & Theorem \ref{thm-divisibility}(3) \\
$\zed_{8} \times \zed_{2}^{2}$ & $8$ & $(3,11,1,2)$ & $+1$ & Theorem \ref{thm-moment}(3) \\
$\zed_{8} \times \zed_{2}^{2}$ & $8$ & $(5,13,2,3)$ & $+1$ & Theorem \ref{thm-moment}(3) \\
$\zed_{4}^{2} \times \zed_{2}$ & $8$ & $(3,11,1,2)$ & $+1$ & Theorem \ref{thm-moment}(3) \\
$\zed_{4}^{2} \times \zed_{2}$ & $8$ & $(5,13,2,3)$ & $+1$ & Theorem \ref{thm-moment}(3) \\
$\zed_{4} \times \zed_{2}^{3}$ & $16$ & $(3,11,1,2)$ & $+1$ & Theorem \ref{thm-moment}(3) \\
$\zed_{4} \times \zed_{2}^{3}$ & $16$ & $(5,13,2,3)$ & $+1$ & Theorem \ref{thm-moment}(3) \\
$\zed_{4} \times \zed_{2}^{3}$ & $16$ & $(7,9,2,1)$ & $-1$ & Theorem \ref{thm-moment}(3) \\
$\zed_{4} \times \zed_{2}^{3}$ & $16$ & $(7,14,3,5)$ & $+2$ & Theorem \ref{thm-moment}(3) \\
$\zed_{4} \times \zed_{2}^{3}$ & $16$ & $(9,14,4,2)$ & $-2$ & Theorem \ref{thm-moment}(3) \\
$\zed_{2}^{5}$ & $32$ & $(3,11,1,2)$ & $+1$ & Corollary \ref{cor-elemabnonexist} \\
$\zed_{2}^{5}$ & $32$ & $(5,13,2,3)$ & $+1$ & Corollary \ref{cor-elemabnonexist} \\
$\zed_{2}^{5}$ & $32$ & $(6,11,2,4)$ & $+2$ & Corollary \ref{cor-elemabnonexist} \\
$\zed_{2}^{5}$ & $32$ & $(7,9,2,1)$ & $-1$ & Corollary \ref{cor-elemabnonexist} \\
$\zed_{2}^{5}$ & $32$ & $(7,14,3,5)$ & $+2$ & Corollary \ref{cor-elemabnonexist} \\
$\zed_{2}^{5}$ & $32$ & $(9,14,4,2)$ & $-2$ & Corollary \ref{cor-elemabnonexist} \\
$\zed_{2}^{5}$ & $32$ & $(10,13,4,6)$ & $+2$ & Corollary \ref{cor-elemabnonexist} \\
$\zed_{2}^{5}$ & $32$ & $(15,15,7,8)$ & $+1$ & Corollary \ref{cor-elemabnonexist} \\
\hline
\end{tabular}
\end{center}
\end{table}

\section{Summary}
\label{summary.sec}


$(\lambda, \mu)$-fold factorizations of groups offer a unified framework for many known and well-studied structures. In this paper, we develop a basic theory for these factorizations, including explicit constructions and necessary conditions for their existence. In a forthcoming paper, computer algorithms are developed and used in conjunction with our existence and nonexistence results  to enumerate all non-equivalent $(\lambda, \mu)$-fold factorizations (with $\lambda \neq \mu$) for all abelian groups of order at most 34.

\section*{AI Usage Statement}

{
The authors utilized ChatGPT to assist in the following:
\begin{enumerate}[label=\textup{(\arabic*)}]
\item To inspire the concept of overlap deviation in Definition~\ref{def-dev}. 
\item To reveal the connection in Section~\ref{sec-zerooffset} between $(\la,\mu)$-fold factorizations of cyclic groups and binary sequence pairs with two-level correlation, and propose an initial version of the product construction in Theorem~\ref{thm-prod}.
\item To verify the data presented in Table~\ref{tab-excluded}.
\end{enumerate}
}

\section*{Acknowledgements}

Shuxing Li’s research is supported by the U.S. National Science Foundation Grant DMS-2452236 and the University of Delaware Research Foundation Strategic Initiative (UDRF-SI) program.

\end{document}